\documentclass[3p,10pt,nopreprintline]{elsarticle}

\usepackage{amsmath}
\usepackage{amsthm}
\usepackage{amssymb}
\usepackage{bm}
\usepackage{graphicx}
\usepackage{epsfig}
\usepackage[normal]{subfigure}
\usepackage{float}
\usepackage{paralist}
\usepackage{caption2}
\usepackage[hidelinks]{hyperref}
\usepackage{enumitem}
\usepackage{aliascnt}
\usepackage[T1]{fontenc}
\usepackage{natbib}
\usepackage{color}
\usepackage{comment}
\usepackage[table]{xcolor}
\usepackage{colortbl,booktabs}

\newtheoremstyle{theorem}{6pt}{6pt}{\itshape}{}{\bfseries}{.}{.5em}{}
\newtheoremstyle{definition}{6pt}{6pt}{\upshape}{}{\bfseries}{.}{.5em}{}

\theoremstyle{theorem}
\newtheorem{theorem}{Theorem}[section]

\newaliascnt{corollary}{theorem}

\aliascntresetthe{corollary}

\newaliascnt{lemma}{theorem}
\newtheorem{lemma}[lemma]{Lemma}
\aliascntresetthe{lemma}

\newaliascnt{sublemma}{theorem}

\aliascntresetthe{sublemma}

\theoremstyle{definition}

\newtheorem{remark}{Remark}[section]

\newaliascnt{proposition}{theorem}
\newtheorem{proposition}[proposition]{Proposition}
\aliascntresetthe{proposition}

\newcommand{\R}{{\mathbb R}}

\newcommand{\bn}{\begin{eqnarray}}
	\newcommand{\en}{\end{eqnarray}}
\newcommand{\bnn}{\begin{eqnarray*}}
	\newcommand{\enn}{\end{eqnarray*}}
\renewcommand{\div}{ {\rm div }  }

\newcommand{\T}{\mathbb{T}}

\newcommand{\al}{\alpha}

\newcommand{\frD}{\mathfrak{D}}

\newcommand{\pt}{\partial_t}

\newcommand{\D}{\nabla}

\usepackage{scalerel,stackengine}
\stackMath
\newcommand\reallywidehat[1]{%
	\savestack{\tmpbox}{\stretchto{%
			\scaleto{%
				\scalerel*[\widthof{\ensuremath{#1}}]{\kern-.6pt\bigwedge\kern-.6pt}%
				{\rule[-\textheight/2]{1ex}{\textheight}}
			}{\textheight}%
		}{0.5ex}}%
	\stackon[1pt]{#1}{\tmpbox}%
}

\numberwithin{equation}{section}

\newcommand{\rhob}{\bar\rho}
\newcommand{\tb}{\bar\theta}
\newcommand{\hb}{\bar H}
\newcommand{\dd}{\,\mathrm d}
\newcommand{\dv}{\operatorname{div}}
\newcommand{\Dt}{D_t}
\hypersetup{
 pdftitle={Global strong solutions for aligned 2.5D non-resistive full MHD at high temperature and magnetic strength},
 pdfauthor={Zixuan Qu and Zhaoyang Shang}
}

\begin{document}
\begin{frontmatter}
\title{Global strong solutions to the aligned 2.5D non-resistive full MHD system
with high temperature and a strong magnetic field}


\author[label1]{Yachun Li}
\address[label1]{School of Mathematical Sciences, MOE-LSC, and SHL-MAC, Shanghai Jiao Tong University, Shanghai 200240, P.R.China;}

\author[label2]{Zixuan Qu}
\address[label2]{School of Mathematical Sciences, Shanghai Jiao Tong University, Shanghai 200240, P.R.China;}

\author[label4]{Zhaoyang Shang\corref{cor1}}
\address[label4]{School of Finance, Shanghai Lixin University of  Accounting and Finance, Shanghai 201209, P.R.China;}

\cortext[cor1]{Corresponding author.}
\ead{shangzhaoyang@sjtu.edu.cn}

\begin{abstract}
We study the Cauchy problem for an aligned vertical-field reduction of
the full compressible non-resistive MHD equations in two space variables.
The shear and bulk viscosities are proportional to $\theta^\alpha$,
and the heat conductivity is proportional to $\theta^\beta$.
For $2<\alpha<\beta<\alpha+2$, we prove global existence and uniqueness
of strong solutions when the background temperature and the comparable
magnetic background are sufficiently large, while the positive density
background is fixed. The velocity and temperature perturbations belong
initially to $H^2$, and the transported magnetic perturbation belongs to
$H^1\cap D^{1,4}$; their absolute norms may be bounded by an arbitrary
fixed constant. Alignment is essential and implies a small density
perturbation relative to the magnetic background.
Our proof combines relative energy estimates, uniform in the background
temperature and magnetic field, with coupled estimates for the material derivatives of the velocity and
temperature. These bounds control the source terms in the damped magnetic-gradient equation, allowing
us to close the global a priori estimates without a separate large-time decay assumption. We construct
the global strong solution by smooth approximation and Sobolev energy estimates for the approximating
solutions, and prove uniqueness by an energy estimate for the difference of two solutions.\end{abstract}
\end{frontmatter}

\section{Introduction}
\label{sec:introduction}
	The full compressible non-resistive MHD equations read:
	\begin{equation}\label{FCNS}
		\begin{cases}
			\pt\rho+\div(\rho u)=0,\\
			\pt(\rho u)+\div(\rho u\otimes u)+\D P=\div\T+\div\left(-\frac{|H|^2}{2}\mathbb I_3+{H\otimes H}\right),\\
			\partial_t \mathcal{E}_{\mathrm{tot}}  + \operatorname{div} \left[ \left( \mathcal{E}_{\mathrm{tot}}  +P + \frac{1}{2} |{H}|^2 \right) {u} - ({u} \cdot {H}){H} \right]
			= \operatorname{div} \left(\T{u} + \tilde{\kappa}\nabla \theta \right),\\
			H_t + {u}\cdot\nabla H - H\cdot\nabla {u} + H\mathrm{div}{u} = 0,\quad \mathrm{div}{H} = 0.
		\end{cases}
	\end{equation}
	where $x=(x_1,x_2,x_3)^{\top}\in\R^3$ and $t\geq 0$ denote the spatial coordinate and time coordinate, $\rho\geq 0$ the mass density, $u = (u^1, u^2, u^3)^{\top}$ the fluid velocity, $H=(H^1,H^2,H^3)^{\top}$ the magnetic field, $P$ the pressure of the fluid, $\theta$ the absolute temperature, $\mathcal{E}_{\mathrm{tot}}  = \rho \left( e + \frac{1}{2} |{u}|^2 \right) + \frac{1}{2} |{H}|^2$ the total energy density, $e$ the specific internal energy. The equation of state for polytropic fluids satisfies
	\begin{equation}\label{Pecv}
		P=R\rho\theta, \quad e=c_v\theta,
	\end{equation}
	where $R$ is the gas constant, and $c_v$ is the specific heat at
	constant volume. $\T$ is the viscosity stress tensor given by $\T=2\tilde{\mu}\frD(u)+\tilde{\lambda}\div u\mathbb{I}_3$, where 
	$$\frD(u) =\frac{\D u+(\D u)^{\top}}{2}$$
	is the deformation tensor, $\mathbb{I}_3$ is the $3 \times 3$ identity matrix, $\tilde{\mu}$ is the shear viscosity coefficient, and $\tilde{\lambda}+\frac{2}{3}\tilde{\mu}$ is the bulk viscosity coefficient. $\tilde{\kappa}$ denotes the coefficient of thermal conductivity.
	In this paper, we will consider the following  temperature-dependent transport coefficients:
	\begin{equation}\label{temperature-dependent}
		\tilde{\mu}(\theta)=\mu\theta^\al, \quad \tilde{\lambda}(\theta)=\lambda\theta^\al, \quad \tilde{\kappa}=\kappa\theta^\beta,
	\end{equation}
	where $\mu$, $\lambda$, $\kappa$, $\al$ and $\beta$ are constants satisfying $\mu>0$, $\kappa>0$, $2\mu+3\lambda\geq0$, $\al>0$, and $\beta>0$.

All unknowns are now assumed to be independent of $x_3$, with
$u=(u^1,u^2,0)$ and the three-dimensional magnetic field
$(0,0,H^3)$. This vertical-field configuration is a special
$2.5$-dimensional reduction of \eqref{FCNS}. From now on,
$x=(x_1,x_2)\in\R^2$, $u=(u^1,u^2)$, and $H=H^3$ denotes a
scalar. The divergence-free constraint is automatically satisfied
by the three-dimensional vector $(0,0,H)$ and does not give an
additional divergence equation for the scalar $H$.
We normalize $c_v=1$. Substituting 
\eqref{Pecv}--\eqref{temperature-dependent} into \eqref{FCNS}
gives
	\begin{equation}\label{FCNS:2}
		\begin{cases}
			\pt\rho+\div(\rho u)=0,\\
			\rho(\pt u+(u\cdot\D)u)+\D P=2\mu\div(\theta^\al\frD(u)) +\lambda\D(\theta^\al\div u)-\frac{1}{2}\nabla(|H|^2),\\
			\rho(\pt\theta+u\cdot\D\theta)+P\div u=2\mu\theta^\al |\frD(u)|^2+\lambda\theta^\al(\div u)^2+\kappa\div(\theta^\beta\D\theta),\\
			H_t+\operatorname{div}(uH)=0.
		\end{cases}
	\end{equation}
	
\paragraph{Local existence with temperature-dependent coefficients}
The dependence of transport coefficients on temperature introduces
additional coupling between the momentum, thermal, and magnetic equations.
Tang and Gao~\citep{TangGao2016} considered full compressible MHD with
viscosity, heat conductivity, and magnetic diffusivity depending on density
and temperature. By maximal regularity and a contraction argument, they
obtained local existence and uniqueness of strong solutions with positive
initial density on bounded domains.
The hyperbolic--parabolic framework of Kawashima~\citep{Kawashima1983}
provides a complementary approach to local solvability near positive
constant states. Applying either framework to the present non-resistive
Cauchy problem requires attention to the absence of magnetic diffusion
and to the different regularities of the transported and parabolic variables.

\paragraph{Global existence in one spatial dimension}
For spatially one-dimensional  flows, substantial large-data results
are available, although restrictions on the constitutive laws are essential.
Hu and Ju~\citep{HuJu2015} established global strong solvability for an
initial-boundary value problem with large data and heat conductivity
$\kappa(\theta)=\theta^q$, $q>0$. This result concerns temperature-dependent
heat conduction; allowing the viscosities to depend on temperature
creates additional difficulties.
Su and Zhao~\citep{SuZhao2018} treated temperature-dependent viscosity, heat
conductivity, and magnetic diffusion, obtaining global classical solutions
and convergence to equilibrium when the adiabatic exponent is sufficiently
close to one. Li and Shang~\citep{LiShang2019} considered an initial-boundary
value problem with transport coefficients proportional to $h(v)\theta^\alpha$,
where $v$ is the specific volume and $h$ satisfies structural assumptions;
large-data global classical solvability was obtained for sufficiently
small $|\alpha|$. For the corresponding Cauchy problem, Shang~\citep{Shang2022}
proved global classical existence and asymptotic stability with
temperature-dependent viscosities and heat conductivity, under smallness
of $|\alpha|$ and sufficiently large constant magnetic diffusivity.
More recently, Song and Zhao~\citep{SongZhao2025} considered power-law
temperature dependence in the viscosity, heat conductivity, and resistivity
on a bounded interval, with a common exponent for the viscosities and
resistivity and a separate exponent for heat conductivity.
They established global non-vacuum strong solutions
and exponential stability for large initial data when the viscosity exponent
is sufficiently small. These are resistive results, and \emph{planar} in these
papers means dependence on one spatial variable, not the two-variable
geometry of \eqref{FCNS:2}.

\paragraph{Global existence in two spatial dimensions}
For the two-dimensional full compressible system, Li and
Shang~\citep{LiShang2022} proved global existence of strong solutions
for an initial-boundary value problem with vacuum and without a
small-data restriction when the viscosity is sufficiently large.
Their viscosities, heat conductivity, and positive magnetic diffusivity
are constant.
This result supplies an important large-dissipation comparison for
the present problem. It does not, however, identify a high-temperature
threshold for the power-law, non-resistive Cauchy system
\eqref{FCNS:2}. In particular, replacing a large constant viscosity
by $\mu\theta^\alpha$ requires control of temperature from below
and of the derivatives of the viscosity.
For temperature-dependent transport in two spatial variables,
Li and Sun~\citep{LiSun2021} obtained global weak solutions for a
vertical-field non-resistive reduction on a bounded planar domain.
We discuss that result in the $2.5$-dimensional paragraph below:
its two-dimensional spatial domain should not be confused with
an unrestricted in-plane vector magnetic field.

\paragraph{Global existence in three spatial dimensions}
Ducomet and Feireisl~\citep{DucometFeireisl2006} established global
finite-energy weak solutions for a radiative full MHD model on bounded
three-dimensional domains with temperature-dependent transport
coefficients and conservative boundary conditions.
For resistive full compressible MHD, Hu and Wang~\citep{HuWang2008}
constructed global variational weak solutions with large finite-energy
data, allowing initial density vacuum, on bounded three-dimensional
domains. Their assumptions permit bounded, temperature-dependent
viscosities with a positive lower bound for the shear viscosity,
heat conductivity comparable to $1+\theta^q$ for $q>2$, and constant
positive magnetic diffusivity; the pressure obeys additional structural
conditions. Feireisl, Gwiazda, Kwon, and
\'{S}wierczewska-Gwiazda~\citep{FeireislEtAl2023} subsequently developed
a global weak theory with non-conservative boundary conditions,
temperature-dependent viscosities, heat conductivity, and magnetic
diffusivity, under structural assumptions on a radiative equation of state.
At the strong-solution level, Zhang~\citep{Zhang2022}
proved a global existence result on a bounded three-dimensional domain
with density- and temperature-dependent viscosities, constant positive
heat conductivity and magnetic diffusivity, and homogeneous Dirichlet
conditions for velocity, temperature, and magnetic field.
The hypotheses include compatibility conditions and a sufficiently
small initial $L^2$ norm of their spatial gradients; initial density
vacuum is allowed. For the full non-resistive system with constant
viscosities and heat conductivity, Li~\citep{Li2022Nonresistive}
proved global well-posedness near a nonzero vertical magnetic equilibrium
in a three-dimensional infinite slab, under a small high-order Sobolev
perturbation assumption. These results concern different combinations
of constitutive laws, domains, and smallness conditions; none directly
provides a large-data, high-temperature theorem for \eqref{FCNS:2}.

\paragraph{Global existence in the special $2.5$-dimensional geometry}
A closely related temperature-dependent predecessor is the weak existence
theorem of Li and Sun~\citep{LiSun2021}. In their model,
the shear viscosity has the form $\mu_0+\mu_1\theta$, the bulk viscosity
in their stress convention is zero, and the heat conductivity has the
form $\kappa_0+\kappa_2\theta^2+\kappa_3\theta^3$, with positive
coefficients. The domain is bounded, and the pressure contains cold
and radiative components in addition to the ideal-gas term.
Their result therefore establishes global solvability for a full,
temperature-dependent, non-resistive model, but not a global strong
solution of the pure power-law ideal-gas Cauchy problem studied here.
For constant transport coefficients, Zhai and Zhang~\citep{ZhaiZhang2025}
obtained global strong solutions near a constant equilibrium in
$\mathbb R^2$ in a hybrid Besov framework, together with decay estimates
under additional low-frequency assumptions. For constant coefficients
on the torus $\mathbb T^2$, Zhai, Li, and Zhao~\citep{ZhaiLiZhao2025}
proved small-perturbation global strong solvability in $H^3$, including
perturbations near an equilibrium with zero magnetic field.
These results distinguish the weak theory for temperature-dependent
coefficients from the perturbative strong theory for constant
coefficients. The present result concerns the
combination of power-law viscosities and heat conductivity, zero
magnetic diffusivity, and the alignment $H_0/\rho_0=\bar H/\bar\rho$.
In this configuration the magnetic pressure contributes to the
compressive part of the dynamics, whereas a high temperature
increases both viscosity and heat conductivity. Establishing global
strong solvability from these large backgrounds requires estimates
that retain their different effects and track all background-dependent
constants.

\subsection{Main result}

Fix $\rhob>0$ and positive constants $c_H\le C_H$. We assume
\begin{equation}\label{eq:fixed-background}
 \tb\ge1,\qquad c_H\tb\le\hb\le C_H\tb,\qquad
 H_0/\rho_0\equiv\hb/\rhob .
\end{equation}
The initial perturbations satisfy, for an arbitrary fixed $C_0>0$,
\begin{equation}\label{eq:fixed-data}
 \|H_0-\hb\|_{H^1\cap D^{1,4}}
       +\|u_0\|_{H^2}+\|\theta_0-\tb\|_{H^2}\le C_0,
 \qquad \rho_0=\frac{\rhob}{\hb}H_0.
\end{equation}
Here $D^{1,4}=\{f\in L^1_{\rm loc}:\nabla f\in L^4\}$, with
$\|f\|_{H^1\cap D^{1,4}}=\|f\|_{H^1}+\|\nabla f\|_4$.
All spaces below are over $\R^2$. Put
\begin{equation}\label{eq:global-damping}
 \nu=2\mu+\lambda,\qquad
 d_*=\frac{\hb^2+R\rhob\tb}{\nu\tb^\alpha}\asymp\tb^{2-\alpha}.
\end{equation}

\begin{theorem}[Global strong solutions at fixed background density]
\label{thm:global}
Let $\mu,\kappa,R>0$, $2\mu+3\lambda\ge0$, and
\begin{equation}\label{eq:fixed-working-exponents}
 2<\alpha<\beta<\alpha+2.
\end{equation}
For every fixed $\rhob>0$, $C_0>0$, and $0<c_H\le C_H$, there
exists $L\ge1$, depending only on these quantities and the physical
parameters and exponents, such that the following holds.
If \eqref{eq:fixed-background}--\eqref{eq:fixed-data} hold and
$\tb\ge L$, the Cauchy problem \eqref{FCNS:2} admits a unique global
strong solution. On every finite interval $[0,T]$ it satisfies
\begin{equation}\label{eq:fixed-target-space}
 \begin{aligned}
 &(\rho-\rhob,H-\hb)\in C([0,T];H^1\cap D^{1,4}),\\
 &(u,\theta-\tb)\in C([0,T];H^2),\\
 &(u_t,\theta_t)\in L^\infty(0,T;L^2)\cap L^2(0,T;H^1),\\
 &(\nabla^2u,\nabla^2\theta)\in L^2(0,T;L^4),\qquad
 (\dot u,\dot\theta)\in C([0,T];L^2)\cap L^2(0,T;H^1).
 \end{aligned}
\end{equation}
The equations hold almost everywhere, the initial data are attained
in the displayed continuous spaces, and
\begin{equation}\label{eq:global-bands}
 H=\frac{\hb}{\rhob}\rho,\qquad
 \frac23\rhob\le\rho\le\frac43\rhob,\quad
 \frac23\tb\le\theta\le\frac43\tb,\quad
 \frac23\hb\le H\le\frac43\hb
 \qquad(t\ge0).
\end{equation}
There is a constant $C$, independent of $T,\tb,\hb$, such that
\begin{equation}\label{eq:global-uniform}
 \sup_{t\ge0}\left[
  \|u(t)\|_{H^2}^2+\|H(t)-\hb\|_{H^1}^2+\|\nabla H(t)\|_4^2
  +\tb^{-1}\|\theta(t)-\tb\|_2^2
  +\|\nabla\theta(t)\|_{H^1}^2\right]\le C,
\end{equation}
and
\begin{equation}\label{eq:global-integrals}
 \begin{aligned}
 \int_0^\infty\Big[&
 \tb^\alpha\|\nabla u\|_2^2+\tb^{\beta-1}\|\nabla\theta\|_2^2
 +d_*\big(\|\nabla H\|_2^2+\|\nabla H\|_4^2\big)
 +\tb^{-\alpha}\big(\|\dot u\|_2^2+\|\nabla\dot u\|_2^2\big)
 +\tb^{-\beta}\big(\|\dot\theta\|_2^2+\|\nabla\dot\theta\|_2^2\big)\\
 &+\tb^\alpha\big(\|\nabla^2u\|_2^2+\|\nabla^2u\|_4^2\big)
 +\tb^\beta\|\nabla^2\theta\|_2^2\Big]\dd t\le C .
 \end{aligned}
\end{equation}
\end{theorem}

\begin{remark}[Meaning of the background assumptions]
Neither $C_0$ nor the absolute velocity and magnetic perturbations
are required to be small. No large-background-density condition is
imposed. However, the threshold may depend on $\rhob^{-1}$; there is
no assertion uniform as $\rhob\downarrow0$. Moreover, alignment implies
\begin{equation}\label{eq:fixed-density-initial}
 \|\rho_0-\rhob\|_{H^1\cap D^{1,4}}
 \le\frac{\rhob}{\hb}C_0.
\end{equation}
Thus the density perturbation is not independently arbitrary.
The theorem is a large-temperature, large-magnetic-background result
for a fixed absolute perturbation bound, not an unrestricted large-data
theorem and not a theorem for general three-dimensional magnetic geometry.
\end{remark}

\begin{remark}[Relation to the existing theory]
The weak temperature-dependent theory of \citep{LiSun2021} and the
constant-coefficient perturbative strong theories of
\citep{ZhaiZhang2025,ZhaiLiZhao2025} address different constitutive
or functional settings. Here both viscosity coefficients vary with
temperature, the heat conductivity has a different power, and the
magnetic equation has no diffusion. The proof tracks the background
powers while the density background is fixed. Its restricted alignment
is essential: it converts magnetic pressure into a density pressure
and supplies the damped magnetic-gradient equation.
No claim of optimality of \eqref{eq:fixed-working-exponents} is made.
\end{remark}

\paragraph{Proof strategy}
The relative energy is bounded independently of $\tb,\hb$ because
alignment makes the density perturbation of size $\hb^{-1}$.
On an arbitrary finite interval, the tentative magnetic damping
integral supplies the elliptic input needed for the first-order
estimates. Coupled material-derivative estimates then improve a
temporary Hessian bound. Their leading constants are independent of
the tentative magnetic bound. In the magnetic-gradient estimate,
the material forcing is therefore a bounded data-dependent term,
not a term that must tend to zero. Choosing the bootstrap number
first and the temperature last closes all the low-order estimates.

To turn these estimates into existence, we use smooth aligned
approximations. A mixed-exponent high-order Sobolev estimate propagates
their regularity using $\nabla\rho\in L^4$, without requiring
$\nabla\rho\in L^\infty$ uniformly in the approximation.
Compactness, a separate time-continuity argument and a difference-energy
estimate yield the strong solution and uniqueness in
\eqref{eq:fixed-target-space}. The original data are not assumed to
belong to a higher Sobolev or a critical Besov space.

Section~\ref{sec:fixed-foundation} establishes the basic energy and
effective flux. Sections~\ref{sec:fixed-derivatives} and
\ref{sec:fdgc} close the global low-order estimates.
Section~\ref{sec:high-persistence} constructs the solution and proves
Theorem~\ref{thm:global}.

\section{Fixed-density initialization and basic energy}
\label{sec:fixed-foundation}
All unmarked norms and spatial integrals are over $\R^2$.
We retain the deformation tensor $\frD(u)$ introduced in Section~\ref{sec:introduction},
now for the planar velocity, with $(\nabla u)_{ij}=\partial_j u^i$.
We write $\dot f=\Dt f=\partial_tf+u\cdot\nabla f$ and use $I$
for the two-dimensional identity matrix.
We also set
\[
 S_0(u)=2\mu\frD(u)+\lambda\dv u\,I,\qquad
 Q(u)=2\mu|\frD(u)|^2+\lambda(\dv u)^2,\qquad
 P_c=R\rho\theta+\frac12H^2 .
\]
Generic constants may depend on the fixed quantities
$\rhob,\rhob^{-1},C_0,\mu,\lambda,\kappa,R,\alpha,\beta,c_H,C_H$.
They never acquire an unstated dependence on $T,\tb,\hb$.
Constants marked $C_M$ may also depend on the specified
bootstrap number $M$.

We repeatedly use the two-dimensional inequalities
\begin{equation}\label{eq:fixed-GN}
 \begin{aligned}
 \|f\|_4^2&\le C\|f\|_2\|\nabla f\|_2,\qquad
 \|f\|_\infty\le C\|f\|_2^{1/3}\|\nabla f\|_4^{2/3},\\
 \|f\|_\infty&\le C\|f\|_2^{1/2}\|\nabla^2f\|_2^{1/2},
 \qquad
 \|f\|_\infty\le C\|f\|_4^{1/2}\|\nabla f\|_4^{1/2}.
 \end{aligned}
\end{equation}
Each inequality is used only when its right-hand side is finite;
see \citep{Gagliardo1959,Nirenberg1959}. These exponents are
two-dimensional and are not replaced by their three-dimensional
counterparts.

\subsection{Alignment, initial positivity and compatibility}
\begin{lemma}[Preservation of alignment]\label{lem:fixed-alignment}
Suppose a positive solution exists on $[0,T]$ with
$\nabla u\in L^1(0,T;L^\infty)$ and with the regularity needed
for the $L^2$ transport estimate. If the initial ratio in
\eqref{eq:fixed-background} is constant, then
\begin{equation}\label{eq:fixed-alignment}
 H=\frac{\hb}{\rhob}\rho,\qquad
 \rho-\rhob=\frac{\rhob}{\hb}(H-\hb),\qquad
 \nabla\rho=\frac{\rhob}{\hb}\nabla H
 \quad\hbox{on }[0,T].
\end{equation}
\end{lemma}
\begin{proof}
For $v=H-(\hb/\rhob)\rho$, the two transport equations imply
$v_t+\dv(uv)=0$ and $v(0)=0$. Consequently
\[
 \frac12\frac{\mathrm d}{\mathrm dt}\|v\|_2^2
 =-\frac12\int(\dv u)v^2\dd x
 \le C\|\nabla u\|_\infty\|v\|_2^2 .
\]
Gronwall's inequality gives $v=0$. The other identities follow
by subtracting or differentiating constants. This proof uses
finite-time transport uniqueness; $\Dt(H/\rho)=0$ alone would
not make a nonconstant initial ratio spatially constant.
\end{proof}

\begin{proposition}[Initial positivity and compatibility at fixed density]
\label{prop:fixed-initial}
Under \eqref{eq:fixed-background}--\eqref{eq:fixed-data},
the initial pointwise bands
\begin{equation}\label{eq:fixed-initial-bands}
 \frac34\rhob\le\rho_0\le\frac54\rhob,\qquad
 \frac34\tb\le\theta_0\le\frac54\tb,\qquad
 \frac34\hb\le H_0\le\frac54\hb
\end{equation}
hold once $\tb$ is sufficiently large depending on $C_0,c_H$.
Define
\begin{equation}\label{eq:fixed-compatibility}
 \begin{aligned}
 g_1={}&\rho_0^{-1/2}
 \big[-\dv(\theta_0^\alpha S_0(u_0))
                   +\nabla(R\rho_0\theta_0+H_0^2/2)\big],\\
 g_2={}&\rho_0^{-1/2}
 \big[-\kappa\dv(\theta_0^\beta\nabla\theta_0)
       -\theta_0^\alpha Q(u_0)+R\rho_0\theta_0\dv u_0\big].
 \end{aligned}
\end{equation}
For $2<\alpha<\beta$, these expressions belong to $L^2$ and
\begin{equation}\label{eq:fixed-initial-material}
 \|g_1\|_2\le C\tb^\alpha,\qquad
 \|g_2\|_2\le C\tb^\beta .
\end{equation}
These bounds initialize the smooth approximations used in the existence proof.
\end{proposition}
\begin{proof}
Equation \eqref{eq:fixed-GN} and $H^2(\R^2)\hookrightarrow L^\infty$
give
\[
 \|H_0-\hb\|_\infty+\|\theta_0-\tb\|_\infty\le CC_0 .
\]
Choose $\tb$ so that $CC_0\le\min\{\tb,\hb\}/4$.
Alignment gives the density band without any lower restriction on
the fixed number $\rhob$.

For the momentum compatibility function, expanding the derivatives
and using $\rho_0=(\rhob/\hb)H_0$ yields
\begin{align*}
 \|g_1\|_2\le C\big[&
 \tb^\alpha\|\nabla^2u_0\|_2+
 \tb^{\alpha-1}\|\nabla\theta_0\|_4\|\nabla u_0\|_4
 +(\tb+\rhob)\|\nabla H_0\|_2+
             \rhob\|\nabla\theta_0\|_2\big].
\end{align*}
The thermal expansion similarly gives
\begin{align*}
 \|g_2\|_2\le C\big[&
 \tb^\beta\|\nabla^2\theta_0\|_2+
 \tb^{\beta-1}\|\nabla\theta_0\|_4^2+
 \tb^\alpha\|\nabla u_0\|_4^2+
             \rhob\tb\|\nabla u_0\|_2\big].
\end{align*}
Here the constants include the factor from $\rho_0^{-1/2}$.
All derivative norms on the right are bounded by
\eqref{eq:fixed-data} and \eqref{eq:fixed-GN}. Since
$\tb\ge1$ and $\beta>\alpha>2$, the claims follow.
If a sufficiently regular solution is constructed,
$\sqrt{\rho_0}\dot u(0)=-g_1$ and
$\sqrt{\rho_0}\dot\theta(0)=-g_2$.
\end{proof}

\subsection{Relative energy and its improved background scale}
The estimates below are first obtained on an interval where
\begin{equation}\label{eq:fixed-bands}
 \frac12\rhob\le\rho\le\frac32\rhob,\qquad
 \frac12\tb\le\theta\le\frac32\tb,\qquad
 H=\frac{\hb}{\rhob}\rho .
\end{equation}
These are tentative pointwise bounds, not conclusions inferred
from local existence for arbitrary time.

\begin{proposition}[Sharp relative energy]\label{prop:fixed-energy}
For a sufficiently regular positive solution satisfying
\eqref{eq:fixed-data} and \eqref{eq:fixed-bands}, define
\begin{equation}\label{eq:fixed-relative-energy}
 \begin{aligned}
 \mathcal E(t)=\int\Big[&
 \tfrac12\rho|u|^2+\tfrac12|H-\hb|^2
 +R\tb\{\rho\log(\rho/\rhob)-\rho+\rhob\}
 +\rho\{\theta-\tb-\tb\log(\theta/\tb)\}\Big]\dd x .
 \end{aligned}
\end{equation}
Then
\begin{equation}\label{eq:fixed-energy-identity}
 \mathcal E'(t)+\int
 [\tb\theta^{\alpha-1}Q(u)
       +\kappa\tb\theta^{\beta-2}|\nabla\theta|^2]\dd x=0,
\end{equation}
and there is a constant $E_*$, independent of $T,\tb,\hb$
and the later bootstrap bounds, such that
\begin{equation}\label{eq:fixed-energy}
 \begin{aligned}
 &\sup_{0\le t\le T}\big(
 \|u(t)\|_2^2+\|H(t)-\hb\|_2^2+
                  \tb^{-1}\|\theta(t)-\tb\|_2^2\big)
+\int_0^T
 [\tb^\alpha\|\nabla u\|_2^2+
                  \tb^{\beta-1}\|\nabla\theta\|_2^2]\dd t
 \le E_* .
 \end{aligned}
\end{equation}
\end{proposition}
\begin{proof}
Spatial cutoffs may be used before passage to the whole space.
Only perturbation energies are integrated. The momentum and
magnetic equations give, after their pressure works cancel,
\begin{equation}\label{eq:fixed-kinmag}
 \frac{\mathrm d}{\mathrm dt}\int
 [\rho|u|^2/2+|H-\hb|^2/2]\dd x
 =R\int(\rho\theta-\rhob\tb)\dv u\dd x
                         -\int\theta^\alpha Q(u)\dd x .
\end{equation}
In particular, the magnetic contribution alone is
$-\frac12\int(H^2-\hb^2)\dv u$, which is finite.
The continuity equation yields
\begin{equation}\label{eq:fixed-density-energy}
 \frac{\mathrm d}{\mathrm dt}
 \int R\tb\{\rho\log(\rho/\rhob)-\rho+\rhob\}\dd x
 =-R\tb\int(\rho-\rhob)\dv u\dd x .
\end{equation}
Testing the thermal equation by $1-\tb/\theta$ gives
\begin{align}\label{eq:fixed-thermal-energy}
 &\frac{\mathrm d}{\mathrm dt}\int
       \rho\{\theta-\tb-\tb\log(\theta/\tb)\}\dd x\notag\\
 &=-R\int\rho(\theta-\tb)\dv u\dd x+
       \int(1-\tb/\theta)\theta^\alpha Q(u)\dd x
       -\kappa\tb\int\theta^{\beta-2}|\nabla\theta|^2\dd x .
\end{align}
Their sum proves \eqref{eq:fixed-energy-identity}.

Taylor's formula on the pointwise bands and alignment imply
\begin{equation}\label{eq:fixed-energy-equivalence}
 \mathcal E(t)\asymp
 \rhob\|u\|_2^2+
       \left(1+\frac{R\rhob\tb}{\hb^2}\right)\|H-\hb\|_2^2+
       \frac{\rhob}{\tb}\|\theta-\tb\|_2^2 .
\end{equation}
Consequently the initial energy is bounded by
\[
 \mathcal E(0)\le CC_0^2
       \left(\rhob+1+\frac{R\rhob\tb}{\hb^2}
                            +\frac{\rhob}{\tb}\right)\le C .
\]
This is the principal improvement over an $O(\rhob\tb)$ bound.

To handle $\lambda<0$, first combine the stresses. With
$\frD_0(u)=\frD(u)-\frac12\dv u\,I$,
\begin{equation}\label{eq:fixed-coercivity}
 Q(u)=2\mu|\frD_0(u)|^2+(\mu+\lambda)(\dv u)^2
       \ge \frac{2\mu}{3}|\frD(u)|^2 .
\end{equation}
Using both temperature bounds and
$\int|\frD(u)|^2=\frac12\int(|\nabla u|^2+|\dv u|^2)$,
the dissipation controls
$c\tb^\alpha\|\nabla u\|_2^2+
c\tb^{\beta-1}\|\nabla\theta\|_2^2$.
Integration proves \eqref{eq:fixed-energy}.
No finite $L^2$ or $L^4$ norm of $H$ itself has been used.
\end{proof}

\subsection{The exact magnetic effective flux}
With $\nu$ as in \eqref{eq:global-damping}, define
\begin{equation}\label{eq:fixed-flux}
 F=\nu\dv u-R(\rho\theta^{1-\alpha}-\rhob\tb^{1-\alpha})
       -\tfrac12(H^2\theta^{-\alpha}-\hb^2\tb^{-\alpha}).
\end{equation}
For $\omega=\partial_1u^2-\partial_2u^1$ and
$\nabla^\perp=(-\partial_2,\partial_1)$, dividing momentum by
$\theta^\alpha$ gives
\[
 \mu\Delta u+(\mu+\lambda)\nabla\dv u
 =\rho\theta^{-\alpha}\dot u+
       \theta^{-\alpha}\nabla P_c
       -\alpha\theta^{-1}S_0(u)\nabla\theta .
\]
Since $\Delta u=\nabla\dv u+\nabla^\perp\omega$, differentiation
of the weighted pressure in \eqref{eq:fixed-flux} yields
\begin{equation}\label{eq:fixed-flux-identity}
 \nabla F+\mu\nabla^\perp\omega
 =\rho\theta^{-\alpha}\dot u+
       \alpha\theta^{-\alpha-1}P_c\nabla\theta
       -\alpha\theta^{-1}S_0(u)\nabla\theta .
\end{equation}
Taking divergence in \eqref{eq:fixed-flux-identity}, we have
\begin{equation}\label{eq:fixed-flux-poisson}
 \begin{aligned}
 \dv(\nabla^\perp\omega)&=-\partial_1\partial_2\omega
                              +\partial_2\partial_1\omega=0,\\
 \Delta F&=\dv\big(\rho\theta^{-\alpha}\dot u
       +\alpha\theta^{-\alpha-1}P_c\nabla\theta
       -\alpha\theta^{-1}S_0(u)\nabla\theta\big).
 \end{aligned}
\end{equation}
For $1<p<\infty$, the Fourier multipliers of
$\partial_i\Delta^{-1}\partial_j$ are
$\xi_i\xi_j/|\xi|^2$.  Thus, for $p=2,4$,
\begin{equation}\label{eq:fixed-flux-projection}
 \begin{aligned}
 \nabla F&=\nabla\Delta^{-1}\dv\big(
       \rho\theta^{-\alpha}\dot u
       +\alpha\theta^{-\alpha-1}P_c\nabla\theta
       -\alpha\theta^{-1}S_0(u)\nabla\theta\big),\\
 \|\nabla F\|_p&\le C_p\big(
       \|\rho\theta^{-\alpha}\dot u\|_p
       +\alpha\|\theta^{-\alpha-1}P_c\nabla\theta\|_p
       +\alpha\|\theta^{-1}S_0(u)\nabla\theta\|_p\big).
 \end{aligned}
\end{equation}
Here the decaying perturbation normalization removes the harmonic
ambiguity.  By alignment and \eqref{eq:fixed-bands},
\begin{equation}\label{eq:fixed-flux-coefficients}
 \begin{aligned}
 |H|\le\tfrac32\hb\le C\tb,\quad
 |P_c|\le C(\tb+\tb^2)\le C\tb^2,\quad
 |\rho\theta^{-\alpha}|\le C\tb^{-\alpha},\\
 |\theta^{-\alpha-1}P_c|\le C\tb^{1-\alpha},
 |\theta^{-1}S_0(u)\nabla\theta|
   \le C\tb^{-1}|\nabla u|\,|\nabla\theta|.
 \end{aligned}
\end{equation}
Inserting \eqref{eq:fixed-flux-coefficients} into
\eqref{eq:fixed-flux-projection} gives
\begin{equation}\label{eq:fixed-flux-estimate}
 \|\nabla F\|_p\le C\big[
 \tb^{-\alpha}\|\dot u\|_p+
 \tb^{-1}\|\nabla u\,\nabla\theta\|_p+
 \tb^{1-\alpha}\|\nabla\theta\|_p\big].
\end{equation}
Every temperature derivative of the magnetic pressure has been
retained. In particular, this is not an estimate obtained by treating
a variable viscosity as constant.

For the velocity estimate, expand the stress and pressure separately:
\begin{equation}\label{eq:fixed-lame-expansion}
 \begin{aligned}
 \dv(\theta^\alpha S_0(u))
   &=\theta^\alpha\big(\mu\Delta u
                 +(\mu+\lambda)\nabla\dv u\big)
       +\alpha\theta^{\alpha-1}S_0(u)\nabla\theta,\\
 \nabla P_c
   &=R\theta\nabla\rho+R\rho\nabla\theta+H\nabla H
   =\left(R\theta\frac{\rhob}{\hb}+H\right)\nabla H
                                            +R\rho\nabla\theta,\\
 \|\theta^{-\alpha}\nabla P_c\|_p
   &\le C\tb^{-\alpha}\big(\tb\|\nabla H\|_p
                                      +\|\nabla\theta\|_p\big).
 \end{aligned}
\end{equation}
Since $\mu>0$ and $\nu>0$, the constant-coefficient
Lam\'e estimate gives
\begin{equation}\label{eq:fixed-lame-elliptic-detail}
 \begin{aligned}
 \|\nabla^2u\|_p
 &\le C_p\|\mu\Delta u+(\mu+\lambda)\nabla\dv u\|_p\\
 &\le C_p\big(\|\rho\theta^{-\alpha}\dot u\|_p
       +\|\theta^{-\alpha}\nabla P_c\|_p
       +\alpha\|\theta^{-1}S_0(u)\nabla\theta\|_p\big).
 \end{aligned}
\end{equation}
Combining \eqref{eq:fixed-flux-coefficients},
\eqref{eq:fixed-lame-expansion} and
\eqref{eq:fixed-lame-elliptic-detail}, we obtain
\begin{equation}\label{eq:fixed-lame}
 \|\nabla^2u\|_p\le C\big[
 \tb^{-\alpha}\|\dot u\|_p+
 \tb^{1-\alpha}\|\nabla H\|_p+
 \tb^{-\alpha}\|\nabla\theta\|_p+
 \tb^{-1}\|\nabla u\,\nabla\theta\|_p\big].
\end{equation}

\section{Global first-order and material-derivative estimates}
\label{sec:fixed-derivatives}

\subsection{First-order estimates on an arbitrary finite interval}
\label{subsec:fd-short-time}

This subsection proves an a priori improvement on an arbitrary finite
interval. The magnetic damping integral is included explicitly in the
tentative bounds; its own improvement is proved in Section~\ref{sec:fdgc}.

Retain the background assumptions \eqref{eq:fixed-background} and the
notation of Section~\ref{sec:fixed-foundation}. Introduce the perturbations
\[
 h=H-\hb,\qquad q=\theta-\tb,
\]
and, for the estimates in this section, write $d=\dv u$.
Magnetic-gradient estimates will be written in terms of $\nabla H$,
which equals $\nabla h$. We consider the system
\begin{equation}\label{eq:fdst-system}
 \begin{cases}
  \Dt\rho+\rho d=0,\qquad \Dt H+Hd=0,\\
  \rho\dot u+\nabla P_c=\operatorname{div}(\theta^\alpha S_0(u)),\\
  \rho\dot\theta+R\rho\theta d
      =\theta^\alpha Q(u)+\kappa\operatorname{div}
                         (\theta^\beta\nabla\theta),
 \end{cases}
\end{equation}
Assume
\begin{equation}\label{eq:fdst-parameters}
 \mu,\kappa,R>0,\qquad 2\mu+3\lambda\ge0,\qquad
 2<\alpha<\beta<\alpha+2 .
\end{equation}
The constants denoted by \(C_*\) below may depend on
\[
 \rhob,\ \rhob^{-1},\ C_0,\ \mu,\lambda,\kappa,R,\alpha,\beta,c_H,C_H,
\]
but are independent of \(\hb,\tb\), the interval length and the
bootstrap constant \(M\).  Constants \(C_M\) may additionally depend
on \(M\).  Neither notation permits an implicit dependence on a
higher norm of a smooth approximation.

\begin{lemma}[Global first-order improvement]
\label{lem:fd-short-time}
Let \(M\ge1\) be fixed and \(T>0\).  In this lemma, $d_*$ denotes
the positive background damping rate defined in
\eqref{eq:global-damping}. It satisfies
\begin{equation}\label{eq:fdst-damping-comparison}
 c\tb^{2-\alpha}\le d_*\le C\tb^{2-\alpha}.
\end{equation}
It is constant in space and time, depends only on the background
state and the physical parameters, and is independent of $M$ and $T$.
It is the value at $(\rho,\theta,H)=(\rhob,\tb,\hb)$ of the
coefficient $(H^2+R\rho\theta)/(\nu\theta^\alpha)$ in the damped
magnetic-gradient equation.  Suppose that a sufficiently
regular positive solution of \eqref{eq:fdst-system} exists on
\([0,T]\), with finite perturbation norms and the prescribed constant
far-field state.  Assume that alignment is preserved,
\begin{equation}\label{eq:fdst-alignment}
 H=\frac{\hb}{\rhob}\rho,\qquad \rho-\rhob=\frac{\rhob}{\hb} h,
\end{equation}
and that
\begin{equation}\label{eq:fdst-bands}
 \frac{\rhob}2\le\rho\le\frac{3\rhob}2,\qquad
 \frac\tb2\le\theta\le\frac{3\tb}2
 \quad\hbox{on }\mathbb R^2\times[0,T].
\end{equation}
The initial data satisfy
\begin{equation}\label{eq:fdst-data}
 \|h_0\|_{H^1\cap D^{1,4}}+\|u_0\|_{H^2}
       +\|q_0\|_{H^2}\le C_0,\qquad
 \rho_0=\rhob+\frac{\rhob}{\hb} h_0 .
\end{equation}
Define
\begin{align}
 \mathcal A_1(T)
   &=\tb^\alpha\sup_{0\le t\le T}\|\nabla u(t)\|_2^2
           +\int_0^T\|\sqrt\rho\,\dot u\|_2^2\,dt,
       \label{eq:fdst-A1}\\
 \mathcal A_2(T)
   &=\tb^\beta\sup_{0\le t\le T}\|\nabla\theta(t)\|_2^2
           +\int_0^T\|\sqrt\rho\,\dot\theta\|_2^2\,dt.
       \label{eq:fdst-A2}
\end{align}
Suppose, explicitly, that the following bootstrap bounds hold:
\begin{equation}\label{eq:fdst-bootstrap}
 \mathcal A_1(T)\le M\tb^\alpha,\qquad
 \mathcal A_2(T)\le M\tb^\beta,\qquad
 \sup_{0\le t\le T}\|\nabla H(t)\|_2^2
          +d_*\int_0^T\|\nabla H\|_2^2\,dt\le M .
\end{equation}
There is a threshold \(\tb_{\mathrm{ell}}(M)\), depending only on
the quantities allowed above, such that, for
\(\tb\ge\tb_{\mathrm{ell}}(M)\),
\begin{equation}\label{eq:fdst-hessian-integrals}
 \int_0^T\|\nabla^2u\|_2^2\,dt\le C_M\tb^{-\alpha},
 \qquad
 \int_0^T\|\nabla^2\theta\|_2^2\,dt\le C_M\tb^{-\beta},
\end{equation}
and
\begin{align}
 \mathcal A_1(T)\le{}&C_*\tb^\alpha
   +C_M\left[
     \tb^{(\alpha+\beta)/2-1}
     +\tb^{(\beta-\alpha)/2}
     +1+\tb^{2-\alpha}\right],
       \label{eq:fdst-A1-improvement}\\
 \mathcal A_2(T)\le{}&C_*\tb^\beta
   +C_M\left[
     \tb^{(\alpha+\beta)/2}
     +\tb^{1+(\beta-\alpha)/2}
     +\tb^{\beta-1}
     +\tb^{(\beta-\alpha)/2}\right].
       \label{eq:fdst-A2-improvement}
\end{align}
Consequently, there are \(K_1,K_2\), independent of \(M,\hb,\tb,T\),
such that, after increasing the threshold to \(\tb_*(M)\),
\begin{equation}\label{eq:fdst-final-improvement}
 \mathcal A_1(T)\le K_1\tb^\alpha,\qquad
 \mathcal A_2(T)\le K_2\tb^\beta .
\end{equation}
No lower bound of the form \(\rhob\ge L\) is required.
\end{lemma}

\begin{proof}
The calculations are performed for sufficiently regular solutions.
Every estimate below uses only the stated data, pointwise bands and
bootstrap bounds.  Spatial integrations may first be made with
cutoffs.  All pressure work is renormalized by the constant far-field
pressure, so no integral of the nonzero background magnetic energy
is used.

\paragraph{The zero-order input}
Proposition~\ref{prop:fixed-energy}, applied on $[0,T]$, gives
\begin{align}\label{eq:fdst-basic}
 &\sup_{0\le t\le T}
       \big(\|u(t)\|_2^2+\|h(t)\|_2^2+
                         \tb^{-1}\|q(t)\|_2^2\big)
 +\tb^\alpha\int_0^T\|\nabla u\|_2^2\,dt
       +\tb^{\beta-1}\int_0^T\|\nabla\theta\|_2^2\,dt
 \le C_* .
\end{align}

\paragraph{Coupled elliptic estimates}
The bootstrap assumptions and the zero-order estimate give
\begin{equation}\label{eq:fdst-time-inputs}
 \begin{gathered}
  \sup_{t\le T}(\|\nabla u\|_2^2+\|\nabla\theta\|_2^2)\le2M,\qquad
  \int_0^T \|\sqrt\rho\,\dot u\|_2^2\,dt\le M\tb^\alpha,\qquad
  \int_0^T \|\sqrt\rho\,\dot\theta\|_2^2\,dt\le M\tb^\beta,\\
  \int_0^T \|\nabla u\|_2^2\,dt\le C_*\tb^{-\alpha},\qquad
  \int_0^T \|\nabla\theta\|_2^2\,dt\le C_*\tb^{1-\beta}.
 \end{gathered}
\end{equation}
Dividing momentum by \(\theta^\alpha\) yields a
constant-coefficient Lam\'e operator:
\[
 \mu\Delta u+(\mu+\lambda)\nabla d
 =\rho\theta^{-\alpha}\dot u+\theta^{-\alpha}\nabla P_c
       -\alpha\theta^{-1}S_0(u)\nabla\theta.
\]
Since alignment implies
\[
 \nabla P_c=\left(R\theta\frac{\rhob}{\hb}+H\right)\nabla H
                   +R\rho\nabla\theta ,
\]
the \(L^2\) Lam\'e estimate, followed by the two-dimensional
inequality \(\|v\|_4^2\le C\|v\|_2\|\nabla v\|_2\), gives
\begin{align}\label{eq:fdst-lame}
 \|\nabla^2u\|_2
 &\le C_*\left[
    \tb^{-\alpha}\|\sqrt\rho\,\dot u\|_2+\tb^{1-\alpha}\|\nabla H\|_2
       +\tb^{-\alpha}\|\nabla\theta\|_2
       +\tb^{-1}\|\nabla u\|_4\|\nabla\theta\|_4\right]\notag\\
 &\le C_*\left[
    \tb^{-\alpha}\|\sqrt\rho\,\dot u\|_2+\tb^{1-\alpha}\|\nabla H\|_2
       +\tb^{-\alpha}\|\nabla\theta\|_2+\tb^{-1}(\|\nabla u\|_2\|\nabla^2u\|_2\|\nabla\theta\|_2\|\nabla^2\theta\|_2)^{1/2}\right]\notag\\
 &\le \tfrac12\|\nabla^2u\|_2+
 C_*\left[
    \tb^{-\alpha}\|\sqrt\rho\,\dot u\|_2+\tb^{1-\alpha}\|\nabla H\|_2
       +\tb^{-\alpha}\|\nabla\theta\|_2+\tb^{-2}\|\nabla u\|_2\|\nabla\theta\|_2\|\nabla^2\theta\|_2\right].
\end{align}
Thus
\begin{equation}\label{eq:fdst-lame-square}
 \|\nabla^2u\|_2^2\le C_*\left[
  \tb^{-2\alpha}\|\sqrt\rho\,\dot u\|_2^2+\tb^{2-2\alpha}\|\nabla H\|_2^2+
  \tb^{-2\alpha}\|\nabla\theta\|_2^2+\tb^{-4}\|\nabla u\|_2^2\|\nabla\theta\|_2^2\|\nabla^2\theta\|_2^2\right].
\end{equation}
Integrating \eqref{eq:fdst-lame-square} in time, and applying these assumptions, we obtain
\begin{equation}\label{eq:fdst-Iu}
 \left(\int_0^T\|\nabla^2u\|_2^2\,dt\right)\le C_M\tb^{-\alpha}+C_M\tb^{-4}\left(\int_0^T\|\nabla^2\theta\|_2^2\,dt\right) .
\end{equation}
Here the magnetic term is estimated by the tentative damping integral:
\[
 \tb^{2-2\alpha}\int_0^T\|\nabla H\|_2^2\,dt
 \le M\tb^{2-2\alpha}d_*^{-1}\le C M\tb^{-\alpha}.
\]
This estimate is valid for arbitrary $T$. The inverse damping scale
is retained explicitly; no large-density factor is used.

Expanding the heat operator gives
\[
 \kappa\theta^\beta\Delta\theta
 =\rho\dot\theta+R\rho\theta d-\theta^\alpha Q(u)
                 -\kappa\beta\theta^{\beta-1}|\nabla\theta|^2.
\]
The \(L^2\) Laplace estimate consequently implies
\begin{align}\label{eq:fdst-heat-elliptic}
 \|\nabla^2\theta\|_2
 &\le C_*\left[
   \tb^{-\beta}\|\sqrt\rho\,\dot\theta\|_2+\tb^{1-\beta}\|\nabla u\|_2
    +\tb^{\alpha-\beta}\|\nabla u\|_4^2
    +\tb^{-1}\|\nabla\theta\|_4^2\right]\notag\\
 &\le C_*\left[
   \tb^{-\beta}\|\sqrt\rho\,\dot\theta\|_2+\tb^{1-\beta}\|\nabla u\|_2
    +\tb^{\alpha-\beta}\|\nabla u\|_2\|\nabla^2u\|_2+\tb^{-1}\|\nabla\theta\|_2\|\nabla^2\theta\|_2\right].
\end{align}
For \(\tb\ge2C_*\sqrt M\), the last term is absorbed.  Squaring
and integrating, and using \eqref{eq:fdst-time-inputs}, gives
\begin{equation}\label{eq:fdst-Itheta}
 \left(\int_0^T\|\nabla^2\theta\|_2^2\,dt\right)\le C_M\tb^{-\beta}
                   +C_M\tb^{2\alpha-2\beta}\left(\int_0^T\|\nabla^2u\|_2^2\,dt\right) .
\end{equation}
In obtaining this inequality the pressure contribution is bounded by
\(C_*\tb^{2-2\beta-\alpha}\le C_*\tb^{-\beta}\).
Substitution of \eqref{eq:fdst-Iu} into \eqref{eq:fdst-Itheta} yields
\[
 \left(\int_0^T\|\nabla^2\theta\|_2^2\,dt\right)\le C_M\bigl(\tb^{-\beta}
                       +\tb^{\alpha-2\beta}\bigr)
              +C_M\tb^{2\alpha-2\beta-4}\left(\int_0^T\|\nabla^2\theta\|_2^2\,dt\right) .
\]
Since \(\beta>\alpha\), the second power in the first bracket is no
larger than \(\tb^{-\beta}\), and the coefficient of \(\left(\int_0^T\|\nabla^2\theta\|_2^2\,dt\right)\)
can be made at most \(1/2\).  This proves both estimates in
\eqref{eq:fdst-hessian-integrals}.  Notice that no estimate for
\(\nabla\dot u\), \(\nabla\dot\theta\), or a pointwise-in-time Hessian
has been used.

\paragraph{The momentum identity}
Set \(\bar P_c=R\rhob\tb+\hb^2/2\) and
\begin{equation}\label{eq:fdst-Psi}
 \Psi(t)=\frac12\int\theta^\alpha Q(u)\,dx
                -\int(P_c-\bar P_c)d\,dx.
\end{equation}
For an integrable scalar $f$, the transport integration rule and the
gradient commutator read
\begin{equation}\label{eq:fdst-transport-rules}
 \begin{aligned}
 \int \Dt f\,dx&=\frac d{dt}\int f\,dx-\int df\,dx,\quad
 \partial_j\dot u^i
   =\Dt\partial_ju^i+(\partial_ju^k)(\partial_ku^i),\\
 \Dt\nabla u&=\nabla\dot u-(\nabla u)^2,\qquad
 \Dt Q(u)=2S_0(u):\big(\nabla\dot u-(\nabla u)^2\big).
 \end{aligned}
\end{equation}
Repeated indices range over $\{1,2\}$ and are summed.  Testing
momentum by $\dot u$ and integrating by parts gives
\begin{equation}\label{eq:fdst-momentum-test-detail}
 \|\sqrt\rho\,\dot u\|_2^2
   =-\int\theta^\alpha S_0(u):\nabla\dot u\,dx
       +\int(P_c-\bar P_c)\dv\dot u\,dx.
\end{equation}
The stress term is
\begin{align}\label{eq:fdst-stress-work-detail}
 -\int\theta^\alpha S_0(u):\nabla\dot u\,dx
 &=-\frac12\int\theta^\alpha \Dt Q(u)\,dx
       -\int\theta^\alpha S_0(u):(\nabla u)^2\,dx\notag\\
 &=-\frac12\frac d{dt}\int\theta^\alpha Q(u)\,dx
       +\frac12\int
          (\alpha\theta^{\alpha-1}\dot\theta+\theta^\alpha d)Q(u)\,dx
        -\int\theta^\alpha S_0(u):(\nabla u)^2\,dx.
\end{align}
For the pressure term, the two transport equations yield
\begin{equation}\label{eq:fdst-pressure-transport-detail}
 \begin{aligned}
 \Dt P_c=R\theta(-\rho d)+R\rho\dot\theta+H(-Hd)
       =R\rho\dot\theta-(R\rho\theta+H^2)d,\\
 \Dt(P_c-\bar P_c)+(P_c-\bar P_c)d
       =R\rho\dot\theta-\left(\tfrac12H^2+\bar P_c\right)d,\quad
 \dv\dot u=\Dt d+\operatorname{tr}((\nabla u)^2).
 \end{aligned}
\end{equation}
Consequently,
\begin{align}\label{eq:fdst-pressure-work-detail}
 \int(P_c-\bar P_c)\dv\dot u\,dx
 &=\frac d{dt}\int(P_c-\bar P_c)d\,dx-\int\big[\Dt(P_c-\bar P_c)+(P_c-\bar P_c)d\big]d\,dx
       +\int(P_c-\bar P_c)\operatorname{tr}((\nabla u)^2)\,dx
       \notag\\
 &=\frac d{dt}\int(P_c-\bar P_c)d\,dx
       -R\int\rho\dot\theta d\,dx
       +\frac12\int H^2d^2\,dx\notag\\
 &\quad+\int P_c\operatorname{tr}((\nabla u)^2)\,dx
       +\bar P_c\int\big[d^2-\operatorname{tr}((\nabla u)^2)\big]dx.
\end{align}
The last integral vanishes, since
\begin{equation}\label{eq:fdst-background-cancellation}
 \begin{aligned}
 \int\operatorname{tr}((\nabla u)^2)\,dx
 &=\int\partial_ju^i\,\partial_iu^j\,dx
   =-\int u^i\partial_i d\,dx
   =\int d^2\,dx.
 \end{aligned}
\end{equation}
Substituting \eqref{eq:fdst-stress-work-detail} and
\eqref{eq:fdst-pressure-work-detail} into
\eqref{eq:fdst-momentum-test-detail}, and using
\eqref{eq:fdst-Psi}, gives
\begin{align}\label{eq:fdst-momentum-identity}
 \Psi'(t)+\|\sqrt\rho\,\dot u\|_2^2={}&
   \frac12\int\bigl(\alpha\theta^{\alpha-1}\dot\theta
                               +\theta^\alpha d\bigr)Q(u)\,dx
       -\int\theta^\alpha S_0(u):(\nabla u)^2\,dx\notag\\
 &-R\int\rho\dot\theta\,d\,dx
       +\frac12\int H^2d^2\,dx
       +\int P_c\,\operatorname{tr}((\nabla u)^2)\,dx .
\end{align}
The last two integrals in \eqref{eq:fdst-momentum-identity} are finite:
they involve bounded coefficients times quadratic velocity gradients.

By \eqref{eq:fdst-basic} and the pointwise bands,
\[
 \|P_c-\bar P_c\|_2
 \le C_*\bigl(\tb\|h\|_2+\|q\|_2\bigr).
\]
For every fixed \(\varepsilon>0\), Young's inequality gives
\begin{equation}\label{eq:fdst-pressure-cross}
 \left|\int(P_c-\bar P_c)d\,dx\right|
 \le\varepsilon\tb^\alpha \|\nabla u\|_2^2+
        C_{*,\varepsilon}\bigl(\tb^{2-\alpha}
                                      +\tb^{1-\alpha}\bigr).
\end{equation}
Together with \eqref{eq:fixed-coercivity}, this proves, after choosing
\(\varepsilon\) sufficiently small,
\begin{equation}\label{eq:fdst-Psi-coercive}
 \Psi(t)\ge c_*\tb^\alpha \|\nabla u\|_2^2-C_*\tb^{2-\alpha},
 \qquad
 |\Psi(0)|\le C_*\tb^\alpha .
\end{equation}

The right-hand side of \eqref{eq:fdst-momentum-identity} is bounded
in absolute value by
\begin{equation}\label{eq:fdst-momentum-remainders}
 C_*\left[
   \tb^{\alpha-1}\|\sqrt\rho\,\dot\theta\|_2\|\nabla u\|_2\|\nabla^2u\|_2
   +\tb^\alpha \|\nabla u\|_2^2\|\nabla^2u\|_2
   +\|\sqrt\rho\,\dot\theta\|_2\|\nabla u\|_2+\tb^2\|\nabla u\|_2^2\right].
\end{equation}
Indeed,
\[
 \begin{gathered}
 \int|\dot\theta||\nabla u|^2
 \le C_*\|\sqrt\rho\,\dot\theta\|_2
                  \|\nabla u\|_2\|\nabla^2u\|_2,\quad
 \|\nabla u\|_3^3\le C\|\nabla u\|_2^2\|\nabla^2u\|_2 .
 \end{gathered}
\]
The gas material term is bounded by \(C_*\|\sqrt\rho\,\dot\theta\|_2\|\nabla u\|_2\).
Finally, \(H^2+|P_c|\le C_*\tb^2\), so the full magnetic
background is retained in the last term of
\eqref{eq:fdst-momentum-remainders}.

Each time integral is controlled separately:
\begin{align} \label{eq:fdst-mom-term1}
 \tb^{\alpha-1}\int_0^T
       \|\sqrt\rho\,\dot\theta\|_2\|\nabla u\|_2
                           \|\nabla^2u\|_2\,dt
&\le \tb^{\alpha-1}\sup_{t\le T}\|\nabla u\|_2
       \left(\int_0^T\|\sqrt\rho\,\dot\theta\|_2^2\,dt\right)^{1/2}\left(\int_0^T\|\nabla^2u\|_2^2\,dt\right)^{1/2}
       \notag\\
 &\le C_M\tb^{(\alpha+\beta)/2-1}, 
       \end{align}
       \begin{align}\label{eq:fdst-mom-term2}
 \tb^\alpha\int_0^T\|\nabla u\|_2^2\|\nabla^2u\|_2\,dt
 \le\tb^\alpha\sup_{t\le T}\|\nabla u\|_2
       \left(\int_0^T\|\nabla u\|_2^2\,dt\right)^{1/2}\left(\int_0^T\|\nabla^2u\|_2^2\,dt\right)^{1/2}
   \le C_M,
              \end{align}
       \begin{align}\label{eq:fdst-mom-term3}
 \int_0^T\|\sqrt\rho\,\dot\theta\|_2\|\nabla u\|_2\,dt
 \le\left(\int_0^T\|\sqrt\rho\,\dot\theta\|_2^2\,dt\right)^{1/2}
       \left(\int_0^T\|\nabla u\|_2^2\,dt\right)^{1/2}
       \le C_M\tb^{(\beta-\alpha)/2},   
       \end{align}
       \begin{align}
 &\tb^2\int_0^T\|\nabla u\|_2^2\,dt
       \le C_*\tb^{2-\alpha}.
       \label{eq:fdst-mom-term4}
\end{align}
Integrate \eqref{eq:fdst-momentum-identity} up to any
\(\tau\in[0,T]\), use \eqref{eq:fdst-Psi-coercive}, and then take
the supremum over \(\tau\).  The same integrated identity at
\(\tau=T\) controls the full integral of \(\|\sqrt\rho\,\dot u\|_2^2\).
Equations \eqref{eq:fdst-mom-term1}--\eqref{eq:fdst-mom-term4}
prove \eqref{eq:fdst-A1-improvement}.

\paragraph{The temperature identity}
The scalar gradient commutator is
\begin{equation}\label{eq:fdst-thermal-gradient-detail}
 \partial_j\dot\theta=\Dt\partial_j\theta
                         +(\partial_ju^k)\partial_k\theta,
 \qquad
 \nabla\dot\theta=\Dt\nabla\theta+(\nabla u)^{\top}\nabla\theta.
\end{equation}
Testing the heat equation by $\dot\theta$ gives
\begin{equation}\label{eq:fdst-thermal-test-detail}
 \|\sqrt\rho\,\dot\theta\|_2^2
 =\int\theta^\alpha Q(u)\dot\theta\,dx
       -R\int\rho\theta d\dot\theta\,dx
       -\kappa\int\theta^\beta\nabla\theta\cdot\nabla\dot\theta\,dx.
\end{equation}
By \eqref{eq:fdst-transport-rules} and
\eqref{eq:fdst-thermal-gradient-detail}, the last term equals
\begin{align}\label{eq:fdst-thermal-diffusion-detail}
 -\kappa\int\theta^\beta\nabla\theta\cdot\nabla\dot\theta\,dx
 &=-\frac\kappa2\int\theta^\beta \Dt|\nabla\theta|^2\,dx
       -\kappa\int\theta^\beta
                (\nabla\theta\otimes\nabla\theta):\nabla u\,dx
       \notag\\
 &=-\frac\kappa2\int \Dt(\theta^\beta|\nabla\theta|^2)\,dx
       +\frac{\kappa\beta}{2}
                 \int\theta^{\beta-1}\dot\theta|\nabla\theta|^2\,dx
 -\kappa\int\theta^\beta
                (\nabla\theta\otimes\nabla\theta):\nabla u\,dx
       \notag\\
 &=-\frac\kappa2\frac d{dt}\int\theta^\beta|\nabla\theta|^2\,dx
       +\frac\kappa2\int\theta^\beta d|\nabla\theta|^2\,dx
       \notag\\
 &\quad+\frac{\kappa\beta}{2}
                 \int\theta^{\beta-1}\dot\theta|\nabla\theta|^2\,dx
       -\kappa\int\theta^\beta
                (\nabla\theta\otimes\nabla\theta):\nabla u\,dx.
\end{align}
Substitution in \eqref{eq:fdst-thermal-test-detail} yields
\begin{align}\label{eq:fdst-temperature-identity}
 \frac\kappa2\frac d{dt}\int\theta^\beta|\nabla\theta|^2\,dx+\|\sqrt\rho\,\dot\theta\|_2^2
 ={}&\int\theta^\alpha Q(u)\dot\theta\,dx
       -R\int\rho\theta d\dot\theta\,dx+\frac{\kappa\beta}{2}
          \int\theta^{\beta-1}\dot\theta|\nabla\theta|^2\,dx\notag\\
 &+\frac\kappa2\int\theta^\beta d|\nabla\theta|^2\,dx -\kappa\int\theta^\beta
          (\nabla\theta\otimes\nabla\theta):\nabla u\,dx .
\end{align}
All convection terms have been included through the transport
integration rule in \eqref{eq:fdst-thermal-diffusion-detail}.

The absolute value of the right-hand side is at most
\[
 \begin{aligned}
 C_*\big[&
   \tb^\alpha\|\sqrt\rho\,\dot\theta\|_2
                      \|\nabla u\|_2\|\nabla^2u\|_2
       +\tb\|\sqrt\rho\,\dot\theta\|_2\|\nabla u\|_2
 +\tb^{\beta-1}\|\sqrt\rho\,\dot\theta\|_2
                      \|\nabla\theta\|_2\|\nabla^2\theta\|_2
       +\tb^\beta\|\nabla u\|_2
                      \|\nabla\theta\|_2\|\nabla^2\theta\|_2\big].
 \end{aligned}
\]
For the last term we used
\(\int|\nabla u||\nabla\theta|^2
\le \|\nabla u\|_2\|\nabla\theta\|_4^2\le C\|\nabla u\|_2\|\nabla\theta\|_2\|\nabla^2\theta\|_2\).
The four corresponding estimates are
\begin{align}\label{eq:fdst-heat-term1}
 &\tb^\alpha\int_0^T\|\sqrt\rho\,\dot\theta\|_2 \|\nabla u\|_2\|\nabla^2u\|_2\,dt \le C_M\tb^{(\alpha+\beta)/2},
\end{align}
\begin{align}\label{eq:fdst-heat-term2}
 &\tb\int_0^T\|\sqrt\rho\,\dot\theta\|_2\|\nabla u\|_2\,dt \le C_M\tb^{1+(\beta-\alpha)/2},
\end{align}
\begingroup
\thinmuskip=1mu\medmuskip=2mu\thickmuskip=2mu
\begin{align}\label{eq:fdst-heat-term3}
 &\tb^{\beta-1}\int_0^T\|\sqrt\rho\,\dot\theta\|_2 \|\nabla\theta\|_2\|\nabla^2\theta\|_2\,dt \le\tb^{\beta-1}\sup_{t\le T}\|\nabla\theta\|_2 \Bigl(\int_0^T\|\sqrt\rho\,\dot\theta\|_2^2\,dt\Bigr)^{1/2}\Bigl(\int_0^T\|\nabla^2\theta\|_2^2\,dt\Bigr)^{1/2} \le C_M\tb^{\beta-1},
\end{align}
\endgroup
\begingroup
\thinmuskip=1mu\medmuskip=2mu\thickmuskip=2mu
\begin{align}\label{eq:fdst-heat-term4}
 &\tb^\beta\int_0^T\|\nabla u\|_2 \|\nabla\theta\|_2\|\nabla^2\theta\|_2\,dt \le\tb^\beta\sup_{t\le T}\|\nabla\theta\|_2 \Bigl(\int_0^T\|\nabla u\|_2^2\,dt\Bigr)^{1/2}\Bigl(\int_0^T\|\nabla^2\theta\|_2^2\,dt\Bigr)^{1/2} \le C_M\tb^{(\beta-\alpha)/2}.
\end{align}
\endgroup
The left-hand gradient energy is comparable to
\(\tb^\beta \|\nabla\theta\|_2^2\), and its initial value is at most
\(C_*\tb^\beta\).  Integration up to \(\tau\le T\), followed
by the same supremum and integral argument as above, proves
\eqref{eq:fdst-A2-improvement}.

\paragraph{Choice of the threshold}
Dividing the two remainder brackets by their respective leading
scales gives
\begin{align*}
 \frac{\mathcal A_1(T)}{\tb^\alpha}
 &\le C_*+C_M\left[
  \tb^{(\beta-\alpha)/2-1}
  +\tb^{(\beta-3\alpha)/2}
  +\tb^{-\alpha}+\tb^{2-2\alpha}\right],\\
 \frac{\mathcal A_2(T)}{\tb^\beta}
 &\le C_*+C_M\left[
  \tb^{(\alpha-\beta)/2}
  +\tb^{1-(\alpha+\beta)/2}
  +\tb^{-1}+\tb^{-(\alpha+\beta)/2}\right].
\end{align*}
Every displayed exponent is strictly negative under
\eqref{eq:fdst-parameters}.  First fix constants \(K_1,K_2\)
larger than the corresponding data-dependent \(C_*\), and then
choose \(\tb_*(M)\) so that the remainder sums fit within these
margins.  This proves \eqref{eq:fdst-final-improvement}.

The argument improves the first-order quantities on the entire interval
$[0,T]$. It uses, but does not yet improve, the magnetic damping bound.
The next subsection controls the material derivatives, and
Section~\ref{sec:fdgc} subsequently improves the magnetic bound.
\end{proof}

\subsection{Global material-derivative and Hessian estimates}
\label{subsec:fd-initial-material}

The following estimate improves a temporary bound for the spatial
Hessians together with the material-derivative energies.  The temporary
Hessian bound is an explicit bootstrap hypothesis, not a conclusion
assumed without proof.  No bootstrap bound for either material
derivative is imposed.

\begin{lemma}[Global material energies and Hessian improvement]
\label{lem:fd-initial-material}
Retain the fixed positive background density, the physical parameters,
the alignment, the pointwise bands and the initial data of
Lemma~\ref{lem:fd-short-time}.  In particular,
\[
 2<\alpha<\beta<\alpha+2,\qquad T>0,\qquad
 c_H\tb\le\hb\le C_H\tb .
\]
Suppose that the first-order estimates have already been improved to
\begin{equation}\label{eq:fdim-first-order}
 \mathcal A_1(T)\le K_1\tb^\alpha,\qquad
 \mathcal A_2(T)\le K_2\tb^\beta,
\end{equation}
where \(K_1,K_2\) depend only on the initial-data bound and the fixed
parameters.  Let \(M\ge1\) be fixed, and assume the temporary bounds
\begin{equation}\label{eq:fdim-bootstrap}
 \sup_{0\le t\le T}\|\nabla H(t)\|_2^2
        +d_*\int_0^T\|\nabla H\|_2^2\,dt\le M,\qquad
 \sup_{0\le t\le T}
       \big(\|\nabla^2u(t)\|_2^2+\|\nabla^2\theta(t)\|_2^2\big)\le M .
\end{equation}
For sufficiently regular solutions there are constants
\(K_{\mathrm{mat}},K_{\mathrm{Hess}}>0\), independent of \(M,\tb,\hb,T\),
and a threshold \(L_{\mathrm{mat}}(M)\), such that
\(\tb\ge L_{\mathrm{mat}}(M)\) implies
\begin{align}\label{eq:fdim-material-result}
 &\sup_{0\le t\le T}
   \left[
     \tb^{-2\alpha}\|\sqrt\rho\,\dot u(t)\|_2^2
       +\tb^{-2\beta}\|\sqrt\rho\,\dot\theta(t)\|_2^2
   \right]+
   \tb^{-\alpha}\int_0^T\|\nabla\dot u\|_2^2\,dt
       +\tb^{-\beta}\int_0^T\|\nabla\dot\theta\|_2^2\,dt
 \le K_{\mathrm{mat}},
\end{align}
and
\begin{equation}\label{eq:fdim-hessian-result}
 \sup_{0\le t\le T}
       \big(\|\nabla^2u(t)\|_2^2+\|\nabla^2\theta(t)\|_2^2\big)
 \le K_{\mathrm{Hess}}.
\end{equation}
The constants may depend on
\(\rhob,\rhob^{-1},C_0,\mu,\lambda,\kappa,R,\alpha,\beta,c_H,C_H,K_1,K_2\).
Thus, if \(M\) was chosen larger than \(2K_{\mathrm{Hess}}\), the
second bootstrap bound in \eqref{eq:fdim-bootstrap} is strictly
improved.  The magnetic-gradient bound remains to be improved
separately. All estimates hold on the same arbitrary finite interval,
with constants independent of its length.
\end{lemma}

\begin{proof}
Throughout the proof \(C_*\) has only the fixed dependencies specified
in the lemma, and \(C_M\) may additionally depend on \(M\).
Every calculation is first made for sufficiently regular solutions;
the estimates contain no uncontrolled higher norm of such a solution.

\paragraph{Previously established inputs and initial material energies}
The first-order estimates, the pointwise density bounds, and
\eqref{eq:fdst-basic} imply
\begin{equation}\label{eq:fdim-inputs}
 \begin{gathered}
  \sup_{t\le T}\big(\|\nabla u\|_2^2+\|\nabla\theta\|_2^2\big)\le C_*,
  \qquad
  \int_0^T\|\dot u\|_2^2\,dt\le C_*\tb^\alpha,\\
  \int_0^T\|\dot\theta\|_2^2\,dt\le C_*\tb^\beta,\qquad
  \int_0^T\|\nabla u\|_2^2\,dt\le C_*\tb^{-\alpha}.
 \end{gathered}
\end{equation}
The coupled elliptic argument in
\eqref{eq:fdst-lame}--\eqref{eq:fdst-Itheta}, with
\eqref{eq:fdim-first-order} and the magnetic-gradient assumption,
also gives
\begin{equation}\label{eq:fdim-hessian-integrals}
 \int_0^T\|\nabla^2u\|_2^2\,dt\le C_M\tb^{-\alpha},
 \qquad
 \int_0^T\|\nabla^2\theta\|_2^2\,dt\le C_M\tb^{-\beta}.
\end{equation}
These estimates precede the material-energy calculation.  In
particular, the bound for
\(\int_0^T\|\dot\theta\|_2^2\,dt\) comes from
\(\mathcal A_2(T)\), not from the material-energy conclusion being
proved.

The equations at \(t=0\), the initial \(H^2\) bounds, and
\(\nabla P_c=(R\theta\rhob/\hb+H)\nabla H+R\rho\nabla\theta\)
give
\begin{align}\label{eq:fdim-initial-u}
 \|\sqrt{\rho_0}\,\dot u_0\|_2
 &\le C_*\big[
   \tb^\alpha\|\nabla^2u_0\|_2
  +\tb^{\alpha-1}\|\nabla\theta_0\|_4\|\nabla u_0\|_4
  +\tb\|\nabla H_0\|_2+\|\nabla\theta_0\|_2\big]
 \le C_*\tb^\alpha ,
\end{align}
and
\begin{align}\label{eq:fdim-initial-theta}
 \|\sqrt{\rho_0}\,\dot\theta_0\|_2
 &\le C_*\big[
   \tb^\beta\|\nabla^2\theta_0\|_2
  +\tb^{\beta-1}\|\nabla\theta_0\|_4^2  +\tb^\alpha\|\nabla u_0\|_4^2
  +\tb\|\nabla u_0\|_2\big]
 \le C_*\tb^\beta .
\end{align}
Only first derivatives of \(H_0\) enter these estimates.

\paragraph{The material momentum identity}
For a vector field \(a\), the conservative material operator satisfies
\begin{equation}\label{eq:fdim-div-commutator}
 (\partial_t+\operatorname{div}(u\,\cdot))\operatorname{div}a
 =\operatorname{div}\big(\Dt a+d\,a-(\nabla u)a\big).
\end{equation}
The corresponding identity for a tensor is applied to each row.
Indeed, in components,
\begin{equation}\label{eq:fdim-div-commutator-detail}
 \begin{aligned}
 (\Dt+d)\partial_j a_j
 &=\partial_j\Dt a_j-(\partial_ju^k)\partial_ka_j
                                      +d\partial_ja_j
 =\partial_j\big(\Dt a_j+da_j-(\partial_ku^j)a_k\big),\\
 (\Dt+d)\partial_j\tau_{ij}
 &=\partial_j\big(\Dt\tau_{ij}+d\tau_{ij}
                                      -\tau_{ik}\partial_ku^j\big).
 \end{aligned}
\end{equation}
In the second equality, the additional second-derivative terms cancel
as $a_j\partial_jd-a_k\partial_kd=0$.
The continuity equation also gives, for a scalar or vector $f$,
\begin{equation}\label{eq:fdim-weighted-material-rule}
 \begin{aligned}
 (\Dt+d)(\rho f)&=\rho \Dt f+(\Dt\rho+\rho d)f=\rho \Dt f,\\
 \int\rho \Dt f\cdot f\,dx
 &=\frac12\int\rho \Dt|f|^2\,dx
   =\frac12\frac d{dt}\int\rho|f|^2\,dx.
 \end{aligned}
\end{equation}
By \eqref{eq:fdst-transport-rules},
\begin{equation}\label{eq:fdim-stress-material-detail}
 \begin{aligned}
 \Dt S_0(u)
 &=S_0(\dot u)-2\mu\operatorname{sym}((\nabla u)^2)
                  -\lambda\operatorname{tr}((\nabla u)^2)I,\\
 \Dt(\theta^\alpha S_0(u))
 &=\theta^\alpha S_0(\dot u)
       +\alpha\theta^{\alpha-1}\dot\theta S_0(u)
-\theta^\alpha\big[2\mu\operatorname{sym}((\nabla u)^2)
                  +\lambda\operatorname{tr}((\nabla u)^2)I\big].
 \end{aligned}
\end{equation}
Here $\operatorname{sym}A=(A+A^{\top})/2$.  Applying
$\Dt+d=\partial_t+\dv(u\,\cdot)$ to the momentum equation gives
\begin{align}\label{eq:fdim-momentum-commuted-detail}
 \rho \Dt\dot u
 &=\dv(\theta^\alpha S_0(\dot u))
       +\alpha\dv(\theta^{\alpha-1}\dot\theta S_0(u))+\dv\Big(\theta^\alpha\big[
       -2\mu\operatorname{sym}((\nabla u)^2)
       -\lambda\operatorname{tr}((\nabla u)^2)I\notag\\
 &\quad+dS_0(u)-S_0(u)(\nabla u)^{\top}\big]\Big)
       -(\Dt+d)\nabla(P_c-\bar P_c).
\end{align}
The pressure contribution, after testing by $\dot u$, is
\begin{align}\label{eq:fdim-pressure-material-detail}
 -\int\dot u\cdot(\Dt+d)\nabla(P_c-\bar P_c)\,dx
 &=\int\big[\Dt P_c+(P_c-\bar P_c)d\big]\dv\dot u\,dx
       -\int(P_c-\bar P_c)\partial_k\dot u^j\partial_ju^k\,dx
       \notag\\
 &=R\int\rho\dot\theta\dv\dot u\,dx
       -\frac12\int H^2d\dv\dot u\,dx
       -\int P_c\partial_k\dot u^j\partial_ju^k\,dx\notag\\
 &\quad+\bar P_c\left(\int\partial_k\dot u^j\partial_ju^k\,dx
                                     -\int d\dv\dot u\,dx\right).
\end{align}
The background terms cancel by
\begin{equation}\label{eq:fdim-pressure-background-detail}
 \int\partial_k\dot u^j\partial_ju^k\,dx
   =-\int\dot u^j\partial_jd\,dx
   =\int d\dv\dot u\,dx.
\end{equation}
Finally,
\begin{equation}\label{eq:fdim-stress-dissipation-detail}
 \int\dot u\cdot\dv(\theta^\alpha S_0(\dot u))\,dx
   =-\int\theta^\alpha S_0(\dot u):\nabla\dot u\,dx
   =-\int\theta^\alpha Q(\dot u)\,dx.
\end{equation}
Testing \eqref{eq:fdim-momentum-commuted-detail} by $\dot u$ and
using \eqref{eq:fdim-weighted-material-rule} and
\eqref{eq:fdim-pressure-material-detail}--\eqref{eq:fdim-stress-dissipation-detail},
we obtain
\begin{align}\label{eq:fdim-momentum-identity}
& \frac12\frac d{dt}\int\rho|\dot u|^2\,dx
       +\int\theta^\alpha Q(\dot u)\,dx\notag\\
 &=
  -\alpha\int\theta^{\alpha-1}\dot\theta\,
                    S_0(u):\nabla\dot u\,dx+\int\theta^\alpha\big[
       2\mu\operatorname{sym}((\nabla u)^2)
       +\lambda\operatorname{tr}((\nabla u)^2)I
       -dS_0(u)+S_0(u)(\nabla u)^{\top}
       \big]:\nabla\dot u\,dx\notag\\
 &\quad+R\int\rho\dot\theta\,\operatorname{div}\dot u\,dx
       -\frac12\int H^2d\,\operatorname{div}\dot u\,dx
       -\int P_c\,\partial_k\dot u^j\,\partial_j u^k\,dx .
\end{align}
All integrals in \eqref{eq:fdim-momentum-identity} are therefore
finite without requiring \(H\in L^2\).

The combined coercivity of \(Q(\dot u)\), including the case
\(\lambda<0\), follows from \eqref{eq:fixed-coercivity}.
Since \(H^2+|P_c|\le C_*\tb^2\), Young's inequality gives
\begin{align}\label{eq:fdim-momentum-inequality}
 &\frac d{dt}\|\sqrt\rho\,\dot u\|_2^2
       +c_*\tb^\alpha\|\nabla\dot u\|_2^2\le C_*\big[
     \tb^{-\alpha}\|\dot\theta\|_2^2
    +\tb^{4-\alpha}\|\nabla u\|_2^2
    +\tb^\alpha\|\nabla u\|_4^4
    +\tb^{\alpha-2}\|\dot\theta\,\nabla u\|_2^2\big].
\end{align}
The power \(\tb^{4-\alpha}\) retains the full magnetic background;
it is not replaced by a smaller power.

After multiplication by \(\tb^{-2\alpha}\), the first three
right-hand terms have time integrals bounded by
\begin{equation}\label{eq:fdim-momentum-easy}
 C_M\big(\tb^{\beta-3\alpha}
                 +\tb^{4-4\alpha}+\tb^{-2\alpha}\big).
\end{equation}
For the third term this uses
\[
 \tb^{-\alpha}\int_0^T\|\nabla u\|_4^4\,dt
 \le C\tb^{-\alpha}\sup_{t\le T}\|\nabla u\|_2^2
                    \int_0^T\|\nabla^2u\|_2^2\,dt .
\]
For the remaining product, two-dimensional interpolation gives,
for every \(\varepsilon>0\),
\begin{align}\label{eq:fdim-momentum-product}
 &\tb^{-\alpha-2}\|\dot\theta\,\nabla u\|_2^2\le C_*\tb^{-\alpha-2}
       \|\dot\theta\|_2\|\nabla\dot\theta\|_2
       \|\nabla u\|_2\|\nabla^2u\|_2\le\varepsilon\tb^{-\beta}\|\nabla\dot\theta\|_2^2
    +C_{*,\varepsilon}\tb^{\beta-2\alpha-4}
       \|\dot\theta\|_2^2\|\nabla u\|_2^2\|\nabla^2u\|_2^2 .
\end{align}
The temporary Hessian bound and \eqref{eq:fdim-inputs} imply
\begin{align}\label{eq:fdim-momentum-product-integral}
 \tb^{\beta-2\alpha-4}
    \int_0^T\|\dot\theta\|_2^2
                 \|\nabla u\|_2^2\|\nabla^2u\|_2^2\,dt\le C_M\tb^{2\beta-2\alpha-4}.
\end{align}
This is where the temporary pointwise Hessian bound is useful.
Using a supremum of \(\|\dot\theta\|_2\) instead would impose an
unnecessary stricter restriction on \(\beta-\alpha\).

\paragraph{The material temperature identity}
First, \eqref{eq:fdst-thermal-gradient-detail} gives
\begin{equation}\label{eq:fdim-heat-flux-expansion}
 \begin{aligned}
 \Dt(\theta^\beta\nabla\theta)
 &=\beta\theta^{\beta-1}\dot\theta\nabla\theta
       +\theta^\beta\big(\nabla\dot\theta
                          -(\nabla u)^{\top}\nabla\theta\big),\\
 \Dt(\theta^\beta\nabla\theta)
       +d\theta^\beta\nabla\theta-(\nabla u)\theta^\beta\nabla\theta
 &=\theta^\beta\nabla\dot\theta
       +\beta\theta^{\beta-1}\dot\theta\nabla\theta
 +\theta^\beta[dI-\nabla u-(\nabla u)^{\top}]\nabla\theta.
 \end{aligned}
\end{equation}
Commuting the heat flux by \eqref{eq:fdim-div-commutator} therefore gives
\begin{align}\label{eq:fdim-heat-flux-identity}
 &(\partial_t+\operatorname{div}(u\,\cdot))
              \operatorname{div}(\theta^\beta\nabla\theta)
 =\operatorname{div}(\theta^\beta\nabla\dot\theta)
       +\beta\operatorname{div}
                     (\theta^{\beta-1}\dot\theta\,\nabla\theta)
 +\operatorname{div}\big(
       \theta^\beta[dI-\nabla u-(\nabla u)^{\top}]
                           \nabla\theta\big).
\end{align}
In particular, all terms involving an undifferentiated \(u\)
cancel in this divergence-form commutator.
The pressure and viscous heating terms expand as
\begin{equation}\label{eq:fdim-thermal-pressure-expansion}
 \begin{aligned}
 (\Dt+d)(R\rho\theta d)
 &=R\rho \Dt(\theta d)
   =R\rho\dot\theta d+R\rho\theta \Dt d
 =R\rho\dot\theta d+R\rho\theta\dv\dot u
                 -R\rho\theta\operatorname{tr}((\nabla u)^2),
 \end{aligned}
\end{equation}
\begin{equation}\label{eq:fdim-thermal-heating-expansion}
 \begin{aligned}
 (\Dt+d)(\theta^\alpha Q(u))
 &=\alpha\theta^{\alpha-1}\dot\theta Q(u)
       +\theta^\alpha \Dt Q(u)+d\theta^\alpha Q(u)\\
 &=\alpha\theta^{\alpha-1}\dot\theta Q(u)
       +2\theta^\alpha S_0(u):\nabla\dot u
+\theta^\alpha[dQ(u)-2S_0(u):(\nabla u)^2].
 \end{aligned}
\end{equation}
Using \eqref{eq:fdim-weighted-material-rule}, the differentiated
temperature equation is thus
\begin{align}\label{eq:fdim-thermal-commuted-detail}
 \rho \Dt\dot\theta-\kappa\dv(\theta^\beta\nabla\dot\theta)
 &=-R\rho\dot\theta d-R\rho\theta\dv\dot u
       +R\rho\theta\operatorname{tr}((\nabla u)^2)\notag\\
 &\quad+\theta^\alpha[dQ(u)-2S_0(u):(\nabla u)^2]
       +\alpha\theta^{\alpha-1}\dot\theta Q(u)
       +2\theta^\alpha S_0(u):\nabla\dot u\notag\\
 &\quad+\kappa\beta\dv(\theta^{\beta-1}\dot\theta\nabla\theta)
       +\kappa\dv\big(\theta^\beta
          [dI-\nabla u-(\nabla u)^{\top}]\nabla\theta\big).
\end{align}
Testing by $\dot\theta$, the three flux terms give
\begin{equation}\label{eq:fdim-thermal-flux-test-detail}
 \begin{aligned}
 -\kappa\int\dot\theta\dv(\theta^\beta\nabla\dot\theta)\,dx
   &=\kappa\int\theta^\beta|\nabla\dot\theta|^2\,dx,\\
 \kappa\beta\int\dot\theta\dv(
                    \theta^{\beta-1}\dot\theta\nabla\theta)\,dx
   &=-\kappa\beta\int\theta^{\beta-1}\dot\theta
                       \nabla\theta\cdot\nabla\dot\theta\,dx,\\
 \kappa\int\dot\theta\dv\big(\theta^\beta
             [dI-\nabla u-(\nabla u)^{\top}]\nabla\theta\big)\,dx
   &=-\kappa\int\theta^\beta
             [dI-\nabla u-(\nabla u)^{\top}]
                          \nabla\theta\cdot\nabla\dot\theta\,dx.
 \end{aligned}
\end{equation}
Together with
$\int\rho \Dt\dot\theta\,\dot\theta\,dx
 =\frac12\frac d{dt}\int\rho|\dot\theta|^2\,dx$,
these identities yield
\begin{align}\label{eq:fdim-temperature-identity}
 &\frac12\frac d{dt}\int\rho|\dot\theta|^2\,dx
       +\kappa\int\theta^\beta|\nabla\dot\theta|^2\,dx
 =-R\int\rho|\dot\theta|^2d\,dx
       -R\int\rho\theta\dot\theta\,
                         \operatorname{div}\dot u\,dx
       +R\int\rho\theta\dot\theta\,
                         \operatorname{tr}((\nabla u)^2)\,dx\notag\\
 &\quad+\int\theta^\alpha\dot\theta
          [dQ(u)-2S_0(u):(\nabla u)^2]\,dx
       +\alpha\int\theta^{\alpha-1}|\dot\theta|^2Q(u)\,dx
       +2\int\theta^\alpha\dot\theta\,
                              S_0(u):\nabla\dot u\,dx\notag\\
 &\quad-\kappa\beta\int\theta^{\beta-1}\dot\theta\,
                              \nabla\theta\cdot\nabla\dot\theta\,dx
       -\kappa\int\theta^\beta
       [dI-\nabla u-(\nabla u)^{\top}]
                              \nabla\theta\cdot\nabla\dot\theta\,dx .
\end{align}
We estimate these eight terms after multiplying the identity by
\(\tb^{-2\beta}\).  Factors depending on the fixed density are
included in \(C_*\).

The first term satisfies
\begin{align}\label{eq:fdim-thermal-1}
 &\tb^{-2\beta}\int|\dot\theta|^2|\nabla u|\,dx 
 \le C\tb^{-2\beta}\|\nabla u\|_2
                         \|\dot\theta\|_2\|\nabla\dot\theta\|_2\le\varepsilon\tb^{-\beta}\|\nabla\dot\theta\|_2^2
       +C_\varepsilon\tb^{-3\beta}
                           \|\nabla u\|_2^2\|\dot\theta\|_2^2 .
\end{align}
The time integral of the last term is at most
\(C_{M,\varepsilon}\tb^{-2\beta}\).
The pressure term involving \(\nabla\dot u\) obeys
\begin{align}\label{eq:fdim-thermal-2}
 &\tb^{1-2\beta}\|\dot\theta\|_2\|\nabla\dot u\|_2 \le\varepsilon\tb^{-\alpha}\|\nabla\dot u\|_2^2 +C_\varepsilon\tb^{2+\alpha-4\beta}\|\dot\theta\|_2^2 ,
\end{align}
whose last term has time integral bounded by
\(C_{M,\varepsilon}\tb^{2+\alpha-3\beta}\).

For the remaining pressure term,
\begingroup
\thinmuskip=1mu\medmuskip=2mu\thickmuskip=2mu
\begin{align}\label{eq:fdim-thermal-3}
 &\tb^{1-2\beta}\int_0^T \int|\dot\theta||\nabla u|^2\,dx\,dt \le C\tb^{1-2\beta}\sup_{t\le T}\|\nabla u\|_2 \Bigl(\int_0^T\|\dot\theta\|_2^2\,dt\Bigr)^{1/2} \Bigl(\int_0^T\|\nabla^2u\|_2^2\,dt\Bigr)^{1/2} \le C_M\tb^{1-(\alpha+3\beta)/2}.
\end{align}
\endgroup
The cubic viscous term is estimated using
\(\|\nabla u\|_6^3\le
C\|\nabla u\|_2\|\nabla^2u\|_2^2\):
\begin{align}\label{eq:fdim-thermal-4}
 \tb^{\alpha-2\beta}\int_0^T
                 \int|\dot\theta||\nabla u|^3\,dx\,dt &\le C\tb^{\alpha-2\beta}
       \sup_{t\le T}
          \big(\|\nabla u\|_2\|\nabla^2u\|_2\big)
       \left(\int_0^T\|\dot\theta\|_2^2\,dt\right)^{1/2}
       \left(\int_0^T\|\nabla^2u\|_2^2\,dt\right)^{1/2}\notag\\
 &\le C_M\tb^{(\alpha-3\beta)/2}.
\end{align}
Thus this term requires no pointwise material-derivative estimate.

For the term containing \(|\dot\theta|^2Q(u)\),
\begin{align}\label{eq:fdim-thermal-5}
 &\tb^{\alpha-1-2\beta}
                      \|\dot\theta\,\nabla u\|_2^2 
 \le C\tb^{\alpha-1-2\beta}
       \|\dot\theta\|_2\|\nabla\dot\theta\|_2
       \|\nabla u\|_2\|\nabla^2u\|_2\le\varepsilon\tb^{-\beta}\|\nabla\dot\theta\|_2^2
       +C_\varepsilon\tb^{2\alpha-2-3\beta}
          \|\dot\theta\|_2^2\|\nabla u\|_2^2\|\nabla^2u\|_2^2 .
\end{align}
The integrated remainder is bounded by
\(C_{M,\varepsilon}\tb^{2\alpha-2-2\beta}\).

The mixed viscous term requires absorption into both dissipations:
\begin{align}\label{eq:fdim-thermal-6}
 \tb^{\alpha-2\beta}
       \int|\dot\theta||\nabla u||\nabla\dot u|\,dx &\le C\tb^{\alpha-2\beta}
      \|\dot\theta\|_2^{1/2}\|\nabla\dot\theta\|_2^{1/2}
      \|\nabla u\|_2^{1/2}\|\nabla^2u\|_2^{1/2}
      \|\nabla\dot u\|_2\notag\\
 &\le\varepsilon\tb^{-\alpha}\|\nabla\dot u\|_2^2
      +C_\varepsilon\tb^{3\alpha-4\beta}
          \|\dot\theta\|_2\|\nabla\dot\theta\|_2
          \|\nabla u\|_2\|\nabla^2u\|_2\notag\\
 &\le\varepsilon\tb^{-\alpha}\|\nabla\dot u\|_2^2
      +\varepsilon\tb^{-\beta}\|\nabla\dot\theta\|_2^2 +C_\varepsilon\tb^{6\alpha-7\beta}
          \|\dot\theta\|_2^2\|\nabla u\|_2^2\|\nabla^2u\|_2^2 .
\end{align}
Its integrated remainder is
\(C_{M,\varepsilon}\tb^{6\alpha-6\beta}\).
The successive uses of Young's inequality have constants depending
on the chosen \(\varepsilon\), which is fixed only after all terms
have been collected.

The derivative of the conductivity produces
\begin{align}\label{eq:fdim-thermal-7}
 \tb^{-\beta-1}
       \int|\dot\theta||\nabla\theta||\nabla\dot\theta|\,dx 
& \le C\tb^{-\beta-1}
       \|\dot\theta\|_2^{1/2}\|\nabla\dot\theta\|_2^{3/2}
       \|\nabla\theta\|_2^{1/2}\|\nabla^2\theta\|_2^{1/2}\notag\\
 &\le\varepsilon\tb^{-\beta}\|\nabla\dot\theta\|_2^2
       +C_\varepsilon\tb^{-\beta-4}
          \|\dot\theta\|_2^2
          \|\nabla\theta\|_2^2\|\nabla^2\theta\|_2^2 .
\end{align}
Here Young's inequality uses conjugate exponents \(4/3\) and \(4\).
The time integral of the last term is bounded by
\(C_{M,\varepsilon}\tb^{-4}\).
Finally, the geometric heat-flux commutator satisfies
\begin{align}\label{eq:fdim-thermal-8}
 &\tb^{-\beta} \int|\nabla u||\nabla\theta||\nabla\dot\theta|\,dx \le\varepsilon\tb^{-\beta}\|\nabla\dot\theta\|_2^2 +C_\varepsilon\tb^{-\beta} \|\nabla u\,\nabla\theta\|_2^2 .
\end{align}
By interpolation and \eqref{eq:fdim-hessian-integrals},
\begin{align}\label{eq:fdim-thermal-8-integral}
 \tb^{-\beta}\int_0^T\|\nabla u\,\nabla\theta\|_2^2\,dt &\le C\tb^{-\beta}\sup_{t\le T} \big(\|\nabla u\|_2\|\nabla\theta\|_2\big) \left(\int_0^T\|\nabla^2u\|_2^2\,dt\right)^{1/2} \left(\int_0^T\|\nabla^2\theta\|_2^2\,dt\right)^{1/2}\notag\\
 &\le C_M\tb^{-(\alpha+3\beta)/2}.
\end{align}

\paragraph{The coupled material-energy improvement}
Choose \(\varepsilon>0\) small relative to the fixed positive
coercivity constants.  Add \(\tb^{-2\alpha}\) times
\eqref{eq:fdim-momentum-inequality} to
\(\tb^{-2\beta}\) times the estimate derived from
\eqref{eq:fdim-temperature-identity}.  All derivative terms with
coefficient \(\varepsilon\) are then absorbed into
\[
 c_*\tb^{-\alpha}\|\nabla\dot u\|_2^2
        +c_*\tb^{-\beta}\|\nabla\dot\theta\|_2^2 .
\]
Integrate up to any \(\tau\le T\), use
\eqref{eq:fdim-initial-u}--\eqref{eq:fdim-initial-theta}, and
then take the supremum in \(\tau\).  The resulting left-hand side
of \eqref{eq:fdim-material-result} is bounded by
\begin{align}\label{eq:fdim-remainder-sum}
 C_*+C_M\big[&
  \tb^{\beta-3\alpha}
  +\tb^{4-4\alpha}
  +\tb^{-2\alpha}
  +\tb^{2\beta-2\alpha-4}
  +\tb^{-2\beta}
  +\tb^{2+\alpha-3\beta}\notag\\
 &+\tb^{1-(\alpha+3\beta)/2}
  +\tb^{(\alpha-3\beta)/2}
  +\tb^{2\alpha-2-2\beta}
  +\tb^{6\alpha-6\beta}
  +\tb^{-4}
  +\tb^{-(\alpha+3\beta)/2}\big].
\end{align}
All exponents in this bracket are strictly negative when
\(2<\alpha<\beta<\alpha+2\).  In particular,
\(2\beta-2\alpha-4<0\) is the only upper-gap restriction introduced
by the material-energy calculation.
For each fixed \(M\), take \(\tb\) large enough that the bracket
multiplied by \(C_M\) is bounded by one.  This proves
\eqref{eq:fdim-material-result} with a constant independent of \(M\).
No material-energy bootstrap or Gronwall estimate was used.

\paragraph{Improvement of the temporary Hessian bound}
The pointwise elliptic estimates proved in the previous
subsection imply
\begin{align}\label{eq:fdim-hessian-u}
 &\|\nabla^2u\|_2 \le C_*\big[ \tb^{-\alpha}\|\sqrt\rho\,\dot u\|_2 +\tb^{1-\alpha}\|\nabla H\|_2 +\tb^{-\alpha}\|\nabla\theta\|_2 +\tb^{-2}\|\nabla u\|_2 \|\nabla\theta\|_2\|\nabla^2\theta\|_2\big],
\end{align}
and, after absorbing
\(C_*\tb^{-1}\|\nabla\theta\|_2\|\nabla^2\theta\|_2\),
\begin{align}\label{eq:fdim-hessian-theta}
 \|\nabla^2\theta\|_2
 \le C_*\big[&
   \tb^{-\beta}\|\sqrt\rho\,\dot\theta\|_2
  +\tb^{1-\beta}\|\nabla u\|_2
  +\tb^{\alpha-\beta}
                     \|\nabla u\|_2\|\nabla^2u\|_2\big].
\end{align}
By \eqref{eq:fdim-first-order} and
\eqref{eq:fdim-material-result}, these become
\begin{align}
 \|\nabla^2u\|_2
 &\le C_*+C_*M^{1/2}\tb^{1-\alpha}
                   +C_*\tb^{-2}\|\nabla^2\theta\|_2,
       \label{eq:fdim-hessian-u-close}\\
 \|\nabla^2\theta\|_2
 &\le C_*+C_*\tb^{\alpha-\beta}\|\nabla^2u\|_2 .
       \label{eq:fdim-hessian-theta-close}
\end{align}
First enlarge the threshold so that
\(M^{1/2}\tb^{1-\alpha}\le1\).  Substituting
\eqref{eq:fdim-hessian-theta-close} into
\eqref{eq:fdim-hessian-u-close}, the remaining coefficient of
\(\|\nabla^2u\|_2\) is \(C_*\tb^{\alpha-\beta-2}\), which can be
made at most \(1/2\).  Both Hessians are therefore uniformly bounded
by a data-dependent constant.  Squaring and taking the time
supremum proves \eqref{eq:fdim-hessian-result}.
\end{proof}

\section{All-time improvement of the magnetic-gradient bootstrap}
\label{sec:fdgc}

The material-energy estimates must enter this argument with a leading
constant independent of the tentative magnetic-gradient bound.  Their
contribution to the squared magnetic-gradient estimate need not tend to
zero as the background temperature increases.  It is instead included
in the data-dependent constant used to choose the bootstrap number.
This distinction removes the need for a separate late-time
material-derivative hypothesis.

Retain the fixed background density and the background assumptions
\eqref{eq:fixed-background}, together with $\nu$ and $d_*$ from
\eqref{eq:global-damping}. The estimates in this section require
\begin{equation}\label{eq:fdgc-regime}
 \alpha>2,\qquad \beta>\alpha.
\end{equation}
The comparison \eqref{eq:fdst-damping-comparison} holds with constants
independent of $\tb,\hb$ and the time interval.  Combined with the material-energy argument, the working
range is $2<\alpha<\beta<\alpha+2$.

\begin{lemma}[All-time magnetic-gradient improvement]
\label{lem:fdgc-improvement}
Let a sufficiently regular solution be given on $\R^2\times[0,T]$,
with finite perturbation norms, and suppose
\begin{equation}\label{eq:fdgc-bands}
 \rho=\frac{\rhob}{\hb}H,\qquad
 \frac{\rhob}{2}\le\rho\le\frac{3\rhob}{2},\qquad
 \frac{\tb}{2}\le\theta\le\frac{3\tb}{2}.
\end{equation}
Assume the basic dissipation bounds
\begin{equation}\label{eq:fdgc-base}
 \int_0^T\|\nabla u\|_2^2\dd t\le E_*\tb^{-\alpha},
 \qquad
 \int_0^T\|\nabla\theta\|_2^2\dd t
       \le E_*\tb^{1-\beta}.
\end{equation}
Define
\begin{equation}\label{eq:fdgc-functional}
 \mathcal C_H(T)=\sum_{p\in\{2,4\}}
 \left(\sup_{0\le t\le T}\|\nabla H(t)\|_p^2
                  +d_*\int_0^T\|\nabla H\|_p^2\dd t\right).
\end{equation}
Let $M\ge1$ be a tentative bootstrap number and suppose
\begin{equation}\label{eq:fdgc-bootstrap}
 \mathcal C_H(T)\le M.
\end{equation}
Suppose the first-order and material-energy argument has already
provided constants $B_*,K_{\theta,*}\ge1$ such that
\begin{equation}\label{eq:fdgc-material-all}
 \int_0^T\left(\|\dot u\|_2^2+
                   \|\nabla\dot u\|_2^2\right)\dd t
       \le B_*\tb^\alpha,
\end{equation}
and
\begin{equation}\label{eq:fdgc-temperature-uniform}
 \sup_{0\le t\le T}
       \left(\|\nabla\theta(t)\|_2^2+
                         \|\nabla^2\theta(t)\|_2^2\right)
       \le K_{\theta,*}.
\end{equation}
Here $E_*,B_*,K_{\theta,*}$ are independent of $M,T,\tb,\hb$.
In addition, assume
\begin{equation}\label{eq:fdgc-temperature-integral}
 \int_0^T\|\nabla^2\theta\|_2^2\dd t
       \le K_\theta(M)\tb^{-\beta},
\end{equation}
where $K_\theta(M)$ is finite and independent of $T,\tb,\hb$.  There are constants $C_*$ independent of $M,T,\tb,\hb$, and
$C_M$ independent of $T,\tb,\hb$, such that
\begin{align}\label{eq:fdgc-lipschitz-result}
 &\int_0^T\left(\|\nabla^2u\|_2^2+
                         \|\nabla^2u\|_4^2\right)\dd t
       \le C_M\tb^{-\alpha},\quad d_*^{-1}\int_0^T\|\nabla u\|_\infty^2\dd t
       \le C_M\tb^{-2},
\end{align}
and
\begin{align}\label{eq:fdgc-global-explicit}
 \mathcal C_H(T)
 &\le C_*\exp\!\big(C_M\tb^{-2}\big)
 \left[\sum_{p\in\{2,4\}}\|\nabla H_0\|_p^2+B_*
       +C_M\big(\tb^{-2}+\tb^{3-\alpha-\beta}\big)\right].
\end{align}
In particular, for $\tb$ sufficiently large depending on $M$,
\begin{equation}\label{eq:fdgc-improved}
 \mathcal C_H(T)
 \le C_*\left(\sum_{p\in\{2,4\}}\|\nabla H_0\|_p^2+B_*\right)
       +C_M\big(\tb^{-2}+\tb^{3-\alpha-\beta}\big),
\end{equation}
after enlarging the constants.  The constants may depend on
$\rhob,\rhob^{-1},E_*,K_{\theta,*}$, the physical parameters and
$\alpha,\beta,c_H,C_H$.  The quantity $C_M$ may additionally
depend on $B_*$ and $K_\theta(M)$.
\end{lemma}

\begin{proof}
All estimates are on the same arbitrary finite interval $[0,T]$.
No restriction $T\le1$, division into early and late times, or
unweighted global $L^1_tL^\infty_x$ estimate for $\nabla u$ is used.

\paragraph{Temperature and material-derivative interpolation}
The two-dimensional Gagliardo--Nirenberg inequalities imply
\begin{align}\label{eq:fdgc-A-theta}
 &\sup_{t\le T}\|\nabla\theta(t)\|_4^4
   \le C\sup_{t\le T}
       \big(\|\nabla\theta\|_2^2\|\nabla^2\theta\|_2^2\big)
   \le C K_{\theta,*}^2,
\end{align}
\begin{align}\label{eq:fdgc-T-four}
 &\int_0^T\|\nabla\theta\|_4^2\dd t
   \le C\left(\int_0^T\|\nabla\theta\|_2^2\dd t\right)^{1/2}
           \left(\int_0^T\|\nabla^2\theta\|_2^2\dd t\right)^{1/2}
   \le C_M\tb^{1/2-\beta},
\end{align}
\begin{align}\label{eq:fdgc-M-four}
 &\int_0^T\|\dot u\|_4^2\dd t
   \le C\left(\int_0^T\|\dot u\|_2^2\dd t\right)^{1/2}
           \left(\int_0^T\|\nabla\dot u\|_2^2\dd t\right)^{1/2}
   \le C B_*\tb^\alpha.
\end{align}
In particular, for $p=2,4$,
\begin{equation}\label{eq:fdgc-temperature-p}
 \int_0^T\|\nabla\theta\|_p^2\dd t
       \le C_M\tb^{1-\beta}.
\end{equation}
The leading constant in \eqref{eq:fdgc-M-four}, unlike the
constant in a tentative magnetic bound, is independent of $M$.

\paragraph{Elliptic estimates retaining the variable-coefficient product}
The constant-coefficient Lam\'e estimate from the momentum equation is
\begin{equation}\label{eq:fdgc-lame-pointwise}
 \|\nabla^2u\|_p\le C\left[
   \tb^{-\alpha}\|\dot u\|_p
   +\tb^{1-\alpha}\|\nabla H\|_p
   +\tb^{-\alpha}\|\nabla\theta\|_p
   +\tb^{-1}\|\nabla u\,\nabla\theta\|_p\right],
 \qquad p=2,4.
\end{equation}
For $p=2$, use
$\|\nabla u\,\nabla\theta\|_2
 \le C\|\nabla\theta\|_4\|\nabla u\|_2^{1/2}
                         \|\nabla^2u\|_2^{1/2}$
and absorb half of $\|\nabla^2u\|_2$.  Squaring and integrating gives
\begin{align}\label{eq:fdgc-U-two}
 \int_0^T\|\nabla^2u\|_2^2\dd t
 &\le C\tb^{-2\alpha}\int_0^T\|\dot u\|_2^2\dd t
       +C\tb^{2-2\alpha}\int_0^T\|\nabla H\|_2^2\dd t\notag\\
 &\quad+C\tb^{-2\alpha}\int_0^T\|\nabla\theta\|_2^2\dd t
       +C\tb^{-4}\left(\sup_{t\le T}\|\nabla\theta(t)\|_4^4\right)
                       \int_0^T\|\nabla u\|_2^2\dd t.
\end{align}
The magnetic term is estimated by the tentative damping integral:
\begin{equation}\label{eq:fdgc-H-contribution}
 \tb^{2-2\alpha}\int_0^T\|\nabla H\|_p^2\dd t
       \le M\tb^{2-2\alpha}d_*^{-1}
       \le CM\tb^{-\alpha},\qquad p=2,4.
\end{equation}
The other terms are bounded by
$C B_*\tb^{-\alpha}$,
$C\tb^{1-\beta-2\alpha}$ and
$C\tb^{-\alpha-4}$, respectively.  Hence
$\int_0^T\|\nabla^2u\|_2^2\dd t\le C_M\tb^{-\alpha}$.

For $p=4$, the inequalities
\[
 \|\nabla u\,\nabla\theta\|_4
       \le\|\nabla\theta\|_4\|\nabla u\|_\infty,
 \qquad
 \|\nabla u\|_\infty
       \le C\|\nabla u\|_4^{1/2}\|\nabla^2u\|_4^{1/2}
\]
allow the same absorption.  After squaring and integrating,
\begin{align}\label{eq:fdgc-U-four}
 \int_0^T\|\nabla^2u\|_4^2\dd t
 &\le C\tb^{-2\alpha}\int_0^T\|\dot u\|_4^2\dd t
       +C\tb^{2-2\alpha}\int_0^T\|\nabla H\|_4^2\dd t
       +C\tb^{-2\alpha}\int_0^T\|\nabla\theta\|_4^2\dd t\notag\\
 &\quad+C\tb^{-4}\left(\sup_{t\le T}\|\nabla\theta(t)\|_4^4\right)
       \left(\int_0^T\|\nabla u\|_2^2\dd t\right)^{1/2}
       \left(\int_0^T\|\nabla^2u\|_2^2\dd t\right)^{1/2}
 \le C_M\tb^{-\alpha}.
\end{align}
Here the last term uses
$\int_0^T\|\nabla u\|_4^2\dd t
 \le C(\int_0^T\|\nabla u\|_2^2\dd t)^{1/2}
       (\int_0^T\|\nabla^2u\|_2^2\dd t)^{1/2}$.
Finally, another two-dimensional interpolation inequality gives
\begin{align}\label{eq:fdgc-J}
 \int_0^T\|\nabla u\|_\infty^2\dd t
 &\le C\left(\int_0^T\|\nabla u\|_2^2\dd t\right)^{1/3}
       \left(\int_0^T\|\nabla^2u\|_4^2\dd t\right)^{2/3}
 \le C_M\tb^{-\alpha}.
\end{align}
Since $d_*^{-1}\le C\tb^{\alpha-2}$, this proves
\eqref{eq:fdgc-lipschitz-result}.

\paragraph{The exact gradient equation}
Use the magnetic effective flux $F$ defined in \eqref{eq:fixed-flux}.
The estimate \eqref{eq:fixed-flux-estimate} reads
\begin{equation}\label{eq:fdgc-F}
 \|\nabla F\|_p\le C\left[
   \tb^{-\alpha}\|\dot u\|_p
   +\tb^{-1}\|\nabla u\,\nabla\theta\|_p
   +\tb^{1-\alpha}\|\nabla\theta\|_p\right],\qquad p=2,4.
\end{equation}
In particular, the temperature derivative of the weighted magnetic
pressure is retained in the last term.
Differentiating $\Dt H+H\dv u=0$ and using alignment gives, for
$w=\nabla H$,
\begin{equation}\label{eq:fdgc-w}
 \Dt w+(\nabla u)^{\top} w+(\dv u)w+a w
       =-\frac H\nu\nabla F-b\nabla\theta,
\end{equation}
where
\[
 a=\frac{H^2+R\rho\theta}{\nu\theta^\alpha},\qquad
 b=\frac H\nu\left((1-\alpha)R\rho\theta^{-\alpha}
                    -\frac\alpha2H^2\theta^{-\alpha-1}\right).
\]
The coefficient bounds follow directly from
\begin{equation}\label{eq:fdgc-damping-coefficients-detail}
 \begin{aligned}
 a&=\frac{H^2+R\rho\theta}{\nu\theta^\alpha}
      \ge c\frac{\hb^2+R\rhob\tb}{\nu\tb^\alpha}=c d_*,\\
 |b|&\le C\hb\big(\rhob\tb^{-\alpha}
                         +\hb^2\tb^{-\alpha-1}\big)
      \le C\tb^{2-\alpha}\le C d_*,\qquad |H|\le C\hb.
 \end{aligned}
\end{equation}
For $p=2,4$, multiply \eqref{eq:fdgc-w} by $|w|^{p-2}w$.
The transport term is
\begin{equation}\label{eq:fdgc-transport-p-detail}
 \int \Dt w\cdot|w|^{p-2}w\,dx
   =\frac1p\frac d{dt}\|w\|_p^p
       -\frac1p\int(\dv u)|w|^p\,dx.
\end{equation}
Hence
\begin{align}\label{eq:fdgc-gradient-p-identity}
 \frac1p\frac d{dt}\|w\|_p^p+\int a|w|^p\,dx
 &=-\int((\nabla u)^{\top} w)\cdot w\,|w|^{p-2}\,dx
       -\left(1-\frac1p\right)\int(\dv u)|w|^p\,dx\notag\\
 &\quad-\int\frac H\nu\nabla F\cdot w\,|w|^{p-2}\,dx
       -\int b\nabla\theta\cdot w\,|w|^{p-2}\,dx.
\end{align}
H\"older's inequality and \eqref{eq:fdgc-damping-coefficients-detail}
give
\begin{align}\label{eq:fdgc-gradient-p-bound}
 \frac1p\frac d{dt}\|w\|_p^p+c d_*\|w\|_p^p
 &\le C\|\nabla u\|_\infty\|w\|_p^p
       +C\hb\|\nabla F\|_p\|w\|_p^{p-1}
          +C d_*\|\nabla\theta\|_p\|w\|_p^{p-1}.
\end{align}
At times when $\|w\|_p>0$, multiply by $2\|w\|_p^{2-p}$ and use
\begin{equation}\label{eq:fdgc-p-to-square-detail}
 \frac2p\|w\|_p^{2-p}\frac d{dt}\|w\|_p^p
    =\frac d{dt}\|w\|_p^2.
\end{equation}
It follows, with fixed positive $c$, that
\begin{align}\label{eq:fdgc-square-before-young}
 \frac d{dt}\|w\|_p^2+2c d_*\|w\|_p^2
 &\le C\|\nabla u\|_\infty\|w\|_p^2
       +C\hb\|\nabla F\|_p\|w\|_p
       +C d_*\|\nabla\theta\|_p\|w\|_p.
\end{align}
Each term on the right is absorbed separately:
\begin{equation}\label{eq:fdgc-young-detail}
 \begin{aligned}
 C\|\nabla u\|_\infty\|w\|_p^2
 &\le\frac c3d_*\|w\|_p^2
       +C d_*^{-1}\|\nabla u\|_\infty^2\|w\|_p^2,\\
 C\hb\|\nabla F\|_p\|w\|_p
 &\le\frac c3d_*\|w\|_p^2
       +C\frac{\hb^2}{d_*}\|\nabla F\|_p^2,\\
 C d_*\|\nabla\theta\|_p\|w\|_p
 &\le\frac c3d_*\|w\|_p^2+C d_*\|\nabla\theta\|_p^2.
 \end{aligned}
\end{equation}
Since $w=\nabla H$, subtraction of the three absorbed terms yields
\begin{align}\label{eq:fdgc-Y}
 \frac{\mathrm d}{\mathrm dt}\|\nabla H\|_p^2
                 +c d_*\|\nabla H\|_p^2
 &\le C d_*^{-1}\|\nabla u\|_\infty^2\|\nabla H\|_p^2
       +C\frac{\hb^2}{d_*}\|\nabla F\|_p^2
       +C d_*\|\nabla\theta\|_p^2.
\end{align}
To include zeros of the norm, take $\delta>0$ and put
$Y_\delta=(\|w\|_p^p+\delta)^{2/p}$.  Multiplying
\eqref{eq:fdgc-gradient-p-bound} by
$2(\|w\|_p^p+\delta)^{2/p-1}$ uses the bounds
\begin{equation}\label{eq:fdgc-zero-regularization-detail}
 \begin{aligned}
 (\|w\|_p^p+\delta)^{2/p-1}\|w\|_p^p
   &\ge Y_\delta-\delta^{2/p},\\
 (\|w\|_p^p+\delta)^{2/p-1}\|w\|_p^p
   &\le Y_\delta,\\
 (\|w\|_p^p+\delta)^{2/p-1}\|w\|_p^{p-1}
   &\le Y_\delta^{1/2}.
 \end{aligned}
\end{equation}
The same Young inequalities therefore give
\begin{align}\label{eq:fdgc-regularized-square-detail}
 Y_\delta'+c d_*Y_\delta
 &\le C d_*^{-1}\|\nabla u\|_\infty^2Y_\delta
       +C\frac{\hb^2}{d_*}\|\nabla F\|_p^2
       +C d_*\|\nabla\theta\|_p^2+C d_*\delta^{2/p}.
\end{align}
Integrating on any subinterval and letting $\delta\downarrow0$
proves \eqref{eq:fdgc-Y} in integrated form, hence almost everywhere
in time, without division by a zero norm.

\paragraph{The full-time source ledger}
Since $\hb^2/d_*\le C\tb^\alpha$, \eqref{eq:fdgc-F} gives
\begin{align}\label{eq:fdgc-source-global-ledger}
 &\frac{\hb^2}{d_*}\int_0^T\|\nabla F\|_p^2\dd t
                  +d_*\int_0^T\|\nabla\theta\|_p^2\dd t\notag\\
 &\quad\le C\left[
   \tb^{-\alpha}\int_0^T\|\dot u\|_p^2\dd t
   +\tb^{\alpha-2}\int_0^T
                         \|\nabla u\,\nabla\theta\|_p^2\dd t
   +\tb^{2-\alpha}\int_0^T\|\nabla\theta\|_p^2\dd t\right].
\end{align}
For either $p=2$ or $p=4$, the first term is bounded by $C B_*$,
by \eqref{eq:fdgc-material-all} and \eqref{eq:fdgc-M-four}.
This term is not asserted to be small.  Its independence of $M$
is exactly what is required for the bootstrap choice below.

For the product term, the uniform thermal Hessian bound gives
\begin{align*}
 \int_0^T\|\nabla u\,\nabla\theta\|_2^2\dd t
 &\le\left(\sup_{t\le T}\|\nabla\theta(t)\|_4^2\right)
                \int_0^T\|\nabla u\|_4^2\dd t
 \le C_M\tb^{-\alpha},\\
 \int_0^T\|\nabla u\,\nabla\theta\|_4^2\dd t
 &\le\left(\sup_{t\le T}\|\nabla\theta(t)\|_4^2\right)
                \int_0^T\|\nabla u\|_\infty^2\dd t
 \le C_M\tb^{-\alpha}.
\end{align*}
Thus its contribution is $C_M\tb^{-2}$.
The final term in \eqref{eq:fdgc-source-global-ledger} is at most
$C_M\tb^{3-\alpha-\beta}$ by
\eqref{eq:fdgc-temperature-p}.  Consequently,
\begin{equation}\label{eq:fdgc-source-global}
 \frac{\hb^2}{d_*}\int_0^T\|\nabla F\|_p^2\dd t
                  +d_*\int_0^T\|\nabla\theta\|_p^2\dd t
 \le C_* B_*+C_M\big(\tb^{-2}+\tb^{3-\alpha-\beta}\big).
\end{equation}
Both remainder powers are strictly negative under
$\beta>\alpha>2$.

\paragraph{Damped Gronwall estimate and independence of the leading constant}
Apply an integrating factor to \eqref{eq:fdgc-Y} on $[0,T]$,
retaining the damping integral on the left.  It follows that
\begin{align*}
 &\sup_{t\le T}\|\nabla H(t)\|_p^2
          +d_*\int_0^T\|\nabla H\|_p^2\dd t\\
 &\le C\exp\!\left(Cd_*^{-1}
                    \int_0^T\|\nabla u\|_\infty^2\dd t\right)
 \left[\|\nabla H_0\|_p^2+B_*
           +C_M\big(\tb^{-2}+\tb^{3-\alpha-\beta}\big)\right].
\end{align*}
Use \eqref{eq:fdgc-lipschitz-result} and sum over $p=2,4$.
This proves \eqref{eq:fdgc-global-explicit}.  Taking the temperature
threshold large enough that the exponential is at most two proves
\eqref{eq:fdgc-improved}.  All bounds are independent of $T$.
\end{proof}

\begin{remark}[A noncircular choice of constants]
\label{rem:fdgc-choices}
Suppose $\sum_{p\in\{2,4\}}\|\nabla H_0\|_p^2\le I_*$.
First determine the improved first-order, material-energy and
thermal-Hessian leading constants from the fixed initial data and
physical parameters.  In particular, $B_*$ and $K_{\theta,*}$
must be determined independently of the tentative magnetic number
$M$, even though the temperature thresholds used to obtain those
improvements may depend on $M$.

Let $C_*$ be the leading constant in
\eqref{eq:fdgc-global-explicit}.  Choose a single $M$ large enough
for the initial values and the other tentative bounds, and also so that
\[
 M\ge4C_*(I_*+B_*+1).
\]
Only after fixing this number choose $\tb$ sufficiently large that
all preceding improvements hold and
\[
 \exp\!\big(C_M\tb^{-2}\big)\le2,\qquad
 C_M\big(\tb^{-2}+\tb^{3-\alpha-\beta}\big)\le1.
\]
Then \eqref{eq:fdgc-global-explicit} improves
$\mathcal C_H(T)\le M$ to $\mathcal C_H(T)\le M/2$.
The appearance of $M$ in the intermediate elliptic estimates is
harmless because it contributes only to terms made small by this
last temperature choice.  If one had only a material bound
$K(M)\tb^\alpha$ instead of \eqref{eq:fdgc-material-all}, this
argument would not by itself close: the non-small source would
then be $C K(M)$.  The preceding independence of the improved
material leading constant is therefore essential.
\end{remark}

\section{Construction, regularity and uniqueness}
\label{sec:high-persistence}

We first justify continuation of smooth approximations, and then
construct the strong solution of Theorem~\ref{thm:global}.  Its constants may depend on the fixed backgrounds
$\rhob,\tb,\hb$.  No smallness or background-uniform estimate is required
at this stage.  In particular, the high Sobolev norms of the approximations
are allowed to depend on the approximation parameter.

Set $c=\hb/\rhob$ and $n=\rho-\rhob$, and retain the temperature
perturbation $q$ introduced in Section~\ref{subsec:fd-short-time}.
The aligned relation $H=c\rho$ reduces the equations exactly to
\begin{align}\label{eq:hp-density}
 \Dt\rho&=-\rho\dv u,
\end{align}
\begin{align}\label{eq:hp-velocity}
 \Dt u&=a(\rho,\theta)\mathcal L u
       +b(\rho,\theta)S_0(u)\nabla\theta
       -\nabla\mathcal W(\rho,\theta)
       +R\log\rho\,\nabla\theta,
\end{align}
\begin{align}\label{eq:hp-temperature}
 \Dt\theta&=k(\rho,\theta)\Delta\theta
       +\ell(\rho,\theta)|\nabla\theta|^2
       +a(\rho,\theta)Q(u)-R\theta\dv u,
\end{align}
where
\begin{gather*}
 \mathcal L u=\mu\Delta u+(\mu+\lambda)\nabla\dv u,\\
 a=\frac{\theta^\alpha}{\rho},\quad
 b=\frac{\alpha\theta^{\alpha-1}}{\rho},\quad
 k=\frac{\kappa\theta^\beta}{\rho},\quad
 \ell=\frac{\kappa\beta\theta^{\beta-1}}{\rho},\qquad
 \mathcal W=R\theta\log\rho+R\theta+c^2\rho.
\end{gather*}
Here $S_0$ and $Q$ are as defined in Section~\ref{sec:fixed-foundation};
we use the identity $Q(u)=S_0(u):\nabla u$.
The tensor--vector product in \eqref{eq:hp-velocity} has components
$\big(S_0(u)\nabla\theta\big)_i
 =\sum_j(S_0(u))_{ij}\partial_j\theta$.
The pressure identity used above is
\begin{equation}\label{eq:hp-pressure-identity}
 \frac1\rho\nabla\left(R\rho\theta+\frac{c^2\rho^2}{2}\right)
 =\nabla\mathcal W-R\log\rho\,\nabla\theta.
\end{equation}
It is important to retain the gradient in this formula: integration by
parts then avoids any need to estimate $\nabla^{m+1}\rho$.

\begin{lemma}[High-order persistence]\label{lem:hp-persistence}
Let $m\ge3$ be an integer and let a smooth aligned solution exist on
$[0,T_*)$, where $T_*<\infty$.  Suppose that, throughout this interval,
\begin{equation}\label{eq:hp-bands}
 0<\rho_*\le\rho\le\rho^*,\qquad
 0<\theta_*\le\theta\le\theta^*,
\end{equation}
and that, for a finite number $K$,
\begin{equation}\label{eq:hp-low-assumption}
 \sup_{t<T_*}
 \left(\|n(t)\|_{H^1}+\|\nabla n(t)\|_4
       +\|u(t)\|_{H^2}+\|q(t)\|_{H^2}\right)\le K,
 \qquad
 \int_0^{T_*}\left(\|\nabla^2u\|_4^2
                        +\|\Dt\theta\|_{H^1}^2\right)\dd t<\infty.
\end{equation}
Then, with
\[
 E_m=\|n\|_{H^m}^2+\|u\|_{H^m}^2+\|q\|_{H^m}^2,
 \qquad
 D_m=\|\nabla u\|_{H^m}^2+\|\nabla q\|_{H^m}^2,
\]
there are constants $c_0>0$ and $C<\infty$ such that
\begin{equation}\label{eq:hp-energy}
 \frac{\dd}{\dd t}E_m+c_0D_m\le C\Phi(t)E_m,
 \qquad
 \Phi=1+\|\nabla u\|_\infty^2+\|\nabla\theta\|_\infty^2
       +\|\nabla^2u\|_4^{4/3}+\|\nabla^2\theta\|_4^{4/3},
\end{equation}
where $\Phi\in L^1(0,T_*)$.
The constants depend only on $m,K$, the bands in \eqref{eq:hp-bands},
the fixed backgrounds, and the physical parameters.  Consequently
\begin{equation}\label{eq:hp-gronwall}
 \sup_{t<T_*}E_m(t)+\int_0^{T_*}D_m\dd t<\infty
 \quad\hbox{whenever }E_m(0)<\infty.
\end{equation}
\end{lemma}

\begin{proof}
We first supply the integrability of the coefficients in
\eqref{eq:hp-energy}.  Write $z=\Dt\theta$.  The thermal equation and the
$L^4(\R^2)$ estimate for the Laplacian imply
\begin{align}
 \|\nabla^2\theta\|_4
 &\le C\big(\|z\|_4+\|\nabla u\|_4
                    +\|\nabla u\|_8^2+\|\nabla\theta\|_8^2\big).
 \label{eq:hp-thermal-L4}
\end{align}
Indeed, one divides
$\kappa\theta^\beta\Delta\theta
 =\rho z+R\rho\theta\dv u-\theta^\alpha Q(u)
  -\kappa\beta\theta^{\beta-1}|\nabla\theta|^2$
by its strictly positive coefficient before applying that estimate.
In two dimensions $H^1$ embeds in every finite $L^p$ space.  The right
side of \eqref{eq:hp-thermal-L4} therefore belongs to $L^2(0,T_*)$ by
\eqref{eq:hp-low-assumption}.  The embedding
$W^{1,4}(\R^2)\hookrightarrow L^\infty(\R^2)$, applied to the gradients,
now gives
\begin{equation}\label{eq:hp-gradient-integrability}
 \nabla u,\nabla\theta\in L^2(0,T_*;L^\infty),
 \qquad \nabla^2u,\nabla^2\theta\in L^2(0,T_*;L^4).
\end{equation}
Thus $\Phi$ is integrable on this finite time interval.

We record the two commutator estimates which permit the use of
$\nabla\rho\in L^4$, rather than $\nabla\rho\in L^\infty$.
For a multi-index $\gamma$ of length $k\ge1$,
\begin{align}\label{eq:hp-parabolic-comm}
 \|[\partial^\gamma,A]\partial_i\partial_kv\|_{4/3}
 &\le C_m\big(
       \|\nabla A\|_4\|\nabla^{k+1}v\|_2
       +\|\nabla^k A\|_2\|\nabla^2v\|_4\big),
\end{align}
\begin{align}\label{eq:hp-transport-comm}
 \|[\partial^\gamma,u\cdot\nabla]\rho\|_2
 &\le C_m\big(
       \|\nabla u\|_\infty\|\nabla^k\rho\|_2
       +\|\nabla^ku\|_4\|\nabla\rho\|_4\big).
\end{align}
For completeness, these integer-order inequalities follow directly from
Leibniz' rule and the Gagliardo--Nirenberg inequalities.  For an
intermediate term with $r$ derivatives on $A$, interpolate between
$\nabla A\in L^4$ and $\nabla^kA\in L^2$, and between
$\nabla^2v\in L^4$ and $\nabla^{k+1}v\in L^2$.
The interpolation parameters are $(r-1)/(k-1)$ and its complement, and
the resulting reciprocal Lebesgue exponents sum to $3/4$.
The product is bounded by the geometric mean of the two terms on the
right of \eqref{eq:hp-parabolic-comm}, and then by their sum.
The endpoint terms, and the case $k=1$, are immediate.
For \eqref{eq:hp-transport-comm}, use instead the endpoint pairs
$(\nabla u,L^\infty),(\nabla^ku,L^4)$ and
$(\nabla\rho,L^4),(\nabla^k\rho,L^2)$; the reciprocal exponents now sum
to $1/2$.  This also proves the estimates componentwise for vector and
tensor functions.

For any smooth function $A=A(\rho,\theta)$ on the compact range in
\eqref{eq:hp-bands}, composition estimates yield
\begin{equation}\label{eq:hp-coefficient-bounds}
 \|\nabla A\|_4\le C_K,
 \qquad
 \|A-A(\rhob,\tb)\|_{H^m}\le C_K E_m^{1/2}.
\end{equation}
The latter tame estimate follows by differentiating $A$, interpolating
each factor between its $L^\infty$ norm and its highest Sobolev
derivative, and applying H\"older's inequality.  Only the bounded ranges
of $\rho,\theta$ enter the composition constant.

Apply $\partial^\gamma$, $0\le|\gamma|\le m$, to
\eqref{eq:hp-density}--\eqref{eq:hp-temperature} and test by the
corresponding derivatives of $n,u,q$.  We describe all types of terms.

\emph{Principal parabolic terms.}
The ellipticity of $\mathcal L$ follows from $\mu>0$ and
$\mu+\lambda\ge\mu/3>0$; the latter is a consequence of
$2\mu+3\lambda\ge0$.
Integration by parts in $a\mathcal L\partial^\gamma u$ and
$k\Delta\partial^\gamma q$ yields positive multiples of
$\|\nabla\partial^\gamma u\|_2^2$ and
$\|\nabla\partial^\gamma q\|_2^2$.
The terms with a derivative on $a$ or $k$ are bounded using
\[
 \|f\|_4\le C\|f\|_2^{1/2}\|\nabla f\|_2^{1/2},
 \qquad
 \|\nabla A\|_4\|\nabla f\|_2\|f\|_4
 \le\varepsilon\|\nabla f\|_2^2
          +C_\varepsilon\|\nabla A\|_4^4\|f\|_2^2.
\]
For the differentiated coefficients, \eqref{eq:hp-parabolic-comm}
and \eqref{eq:hp-coefficient-bounds} give, after summation,
\begin{align*}
 \left|\sum_{1\le|\gamma|\le m}
  \int[\partial^\gamma,A]\nabla^2v\,
                 \partial^\gamma v\dd x\right|
 &\le\varepsilon D_m
    +C_\varepsilon\big(1+\|\nabla^2v\|_4^{4/3}\big)E_m,
 \qquad v=u\ \hbox{or}\ q.
\end{align*}
The power $4/3$ is obtained from
$\|\nabla^2v\|_4 E_m^{3/4}D_m^{1/4}$ by Young's inequality.

\emph{Transport and density terms.}
Integration by parts in the undifferentiated transport operator produces
at most $C\|\dv u\|_\infty E_m$.
For the density commutator, \eqref{eq:hp-transport-comm} yields
\[
 C\|\nabla u\|_\infty E_m
 +C\|\nabla\rho\|_4 E_m^{3/4}D_m^{1/4}
 \le\varepsilon D_m+C_K(1+\|\nabla u\|_\infty)E_m.
\]
The derivatives of $\rho\dv u$ are bounded by
$C\|\rho\|_\infty\|\nabla u\|_{H^m}
 +C\|\dv u\|_\infty\|n\|_{H^m}$.
The velocity and temperature transport commutators are bounded by
$C(\|\nabla u\|_\infty+\|\nabla\theta\|_\infty)E_m$.
These estimates include the zeroth-order equations, treated directly.

\emph{Pressure.}
By \eqref{eq:hp-pressure-identity}, all the terms from
$-\nabla\mathcal W$ have the form
\[
 \int\partial^\gamma(\mathcal W-\mathcal W(\rhob,\tb))\,
                   \dv\partial^\gamma u\dd x,
\]
and are bounded by $\varepsilon D_m+C_\varepsilon E_m$.
The remaining pressure term satisfies
\[
 \|\log\rho\,\nabla\theta\|_{H^m}
 \le C\|\nabla q\|_{H^m}
          +C\|\nabla\theta\|_\infty E_m^{1/2},
\]
and hence contributes at most
$\varepsilon D_m+C_\varepsilon(1+\|\nabla\theta\|_\infty^2)E_m$.

\emph{Gradient products and thermal pressure.}
All remaining nonlinearities are of the forms
$B(\rho,\theta)\nabla v\nabla w$, with $v,w\in\{u,q\}$,
or $\theta\dv u$.
The ordinary tame product estimate, together with
\eqref{eq:hp-coefficient-bounds}, gives
\begin{align*}
 \|B\nabla v\nabla w\|_{H^m}
 &\le C_K\big(\|\nabla v\|_\infty\|\nabla w\|_{H^m}
             +\|\nabla w\|_\infty\|\nabla v\|_{H^m}
             +\|\nabla v\|_\infty\|\nabla w\|_\infty E_m^{1/2}\big).
\end{align*}
After testing and Young's inequality, the contribution is bounded by
\[
 \varepsilon D_m+
 C_\varepsilon(\|\nabla u\|_\infty^2
                    +\|\nabla\theta\|_\infty^2)E_m.
\]
Here the product of the two gradient $L^\infty$ norms is bounded by
half the sum of their squares; no fourth power in time is required.
Finally,
$\|\theta\dv u\|_{H^m}
 \le C\|\nabla u\|_{H^m}
       +C\|\dv u\|_\infty\|q\|_{H^m}$,
which is of the same admissible type.
Taking the finitely many $\varepsilon$'s sufficiently small proves
\eqref{eq:hp-energy}.  Gronwall's inequality and
\eqref{eq:hp-gradient-integrability} prove \eqref{eq:hp-gronwall}.
\end{proof}

\subsection{Smooth seeds and the role of approximation}
\label{subsec:hp-seed}

The classical Cauchy theory for quasilinear hyperbolic--parabolic systems
provides smooth seeds.  To apply the theorem of Vol'pert and Khudyaev,
\citep[Theorem~1 and Remark~1]{VolpertKhudyaev1972},
take $(u,\theta)$ as the parabolic variables and $\rho$ as the scalar
hyperbolic variable.  In the equations before division by $\rho$, the
time matrices are positive on \eqref{eq:hp-bands}.  The hyperbolic spatial
matrices are scalar and therefore symmetric.  The two principal
parabolic blocks are the strongly elliptic Lam\'e operator multiplied
by $\theta^\alpha$ and the positive scalar diffusion
$\kappa\theta^\beta\Delta$.  Their leading symbols have no coupling
with each other.  The pressure gradient and the gradient products are
permitted lower-order terms, smooth on the positive phase region.
In the shifted variables $(n,u,q)$, this region is given by
$n>-\rhob$ and $q>-\tb$, with unrestricted velocity and
first-derivative coordinates; it is convex and contains the origin.
The lower-order terms vanish at the constant equilibrium.
For completeness, if the state and its first derivatives range over
a compact subset $\Omega'$ of this region, integration by parts gives
\begin{align}\label{eq:hp-garding}
 -\int\phi\cdot\theta^\alpha\mathcal L\phi\,dx
 -\kappa\int\psi\theta^\beta\Delta\psi\,dx
 &\ge c_{\Omega'}\big(\|\nabla\phi\|_2^2+\|\nabla\psi\|_2^2\big)
       -C_{\Omega'}\big(\|\phi\|_2^2+\|\psi\|_2^2\big)
\end{align}
for smooth compactly supported test functions $(\phi,\psi)$.
Indeed, the principal terms are coercive since
$\mu+\lambda\ge\mu/3>0$; the remaining terms contain
$\nabla\theta^\alpha$ or $\nabla\theta^\beta$, which are bounded
on $\Omega'$, and are absorbed by Young's inequality.
This verifies the variable-coefficient parabolic coercivity required
by the smooth Cauchy theory, locally around every compact positive range.

For precision, the general theorem in that source uses an integer
$l>d/2+3$; its quasilinear refinement uses $l>d/2+2$.
We may simply mollify the initial perturbations to smooth data belonging
to $H^5(\R^2)$ (indeed to every $H^m$), and use the general theorem.
This choice is made only for the approximations and does not impose
$H^5$, or $H^3$, on the original data.  Lemma~\ref{lem:hp-persistence}
applies at $m=5$.  It prevents a finite loss of the high Sobolev norm
whenever the positive bands and the low-order estimates are maintained.
The classical continuation theorem therefore extends each smooth
approximation on every such finite interval.  The initial $E_5$ may
diverge along the approximating sequence; this causes no difficulty
because compactness uses the uniform low-order estimates, not
\eqref{eq:hp-gronwall} uniformly in the sequence.

Positive mollification preserves the initial density and temperature
bands.  Defining $H_0^{(\varepsilon)}=c\rho_0^{(\varepsilon)}$ preserves
the alignment exactly.  The initial material derivatives, computed
from the equations, converge in $L^2$ under convergence of
$n_0$ in $H^1\cap W^{1,4}$ and of $(u_0,q_0)$ in $H^2$:
the only extra products are products of first derivatives in $L^4$.
Hence their $L^2$ bounds require no additional compatibility hypothesis
when the density is bounded away from zero.

The remaining subsections supply the simultaneous bootstrap, compactness,
strong time continuity and uniqueness arguments.

\subsection{Simultaneous bootstrap and global smooth approximations}
\label{subsec:global-bootstrap}

We now prove Theorem~\ref{thm:global}. Start with smooth aligned data
and their classical local solution from
\citep[Theorem~1 and Remark~1]{VolpertKhudyaev1972}.
Use the pointwise bands \eqref{eq:fixed-bands} and the tentative bounds
\begin{equation}\label{eq:global-bootstrap}
 \frac{\mathcal A_1(T)}{\tb^\alpha}\le M,\qquad
 \frac{\mathcal A_2(T)}{\tb^\beta}\le M,\qquad
 \sup_{t\le T}\big(\|\nabla^2u(t)\|_2^2+
                         \|\nabla^2\theta(t)\|_2^2\big)\le M,\qquad
 \mathcal C_H(T)\le M .
\end{equation}
Here the functionals are defined in
\eqref{eq:fdst-A1}--\eqref{eq:fdst-A2} and
\eqref{eq:fdgc-functional}. All four are continuous in the endpoint
for a smooth solution.

First determine the data-dependent constants $K_1,K_2$ in
Lemma~\ref{lem:fd-short-time}, and then $K_{\rm mat},K_{\rm Hess}$ in
Lemma~\ref{lem:fd-initial-material}. The constants do not depend on
$M$; only the thresholds used to obtain them do.
The latter estimates imply the inputs of
Lemma~\ref{lem:fdgc-improvement} with, for example,
\[
 B_*=1+2K_1/\rhob+K_{\rm mat},\qquad
 K_{\theta,*}=1+K_2+K_{\rm Hess}.
\]
The integrated thermal Hessian bound comes from
\eqref{eq:fdst-hessian-integrals}, with its permitted $M$-dependent
constant. Choose $M$ large enough to contain the initial values and
four times each improved leading bound, including
$C_*(2C_0^2+B_*+1)$; the data bound gives
$\|\nabla H_0\|_2^2+\|\nabla H_0\|_4^2\le2C_0^2$.
Only then choose $\tb$ larger than every threshold in the three
lemmas. The four quantities in \eqref{eq:global-bootstrap} are
thereby at most $M/2$.

The pointwise bounds improve as well. The basic energy,
alignment and the magnetic-gradient bound give
\begin{equation}\label{eq:global-pointwise-density}
 \frac{\|\rho-\rhob\|_\infty}{\rhob}
 =\frac{\|H-\hb\|_\infty}{\hb}
 \le\frac C{\hb}\|H-\hb\|_2^{1/3}\|\nabla H\|_4^{2/3}
 \le C_M\tb^{-1}.
\end{equation}
The improved, unweighted thermal Hessian bound gives
\begin{equation}\label{eq:global-pointwise-temperature}
 \frac{\|\theta-\tb\|_\infty}{\tb}
 \le C\tb^{-1}\|\theta-\tb\|_2^{1/2}
                         \|\nabla^2\theta\|_2^{1/2}
 \le C_M\tb^{-3/4}.
\end{equation}
Increasing $\tb$ once more makes both relative errors at most $1/3$.
This proves the strict bands in \eqref{eq:global-bands}.
The initial bands are even stricter by Proposition~\ref{prop:fixed-initial}.

These choices are independent of the time endpoint and the high norms
of a smooth approximation. A first exit from any tentative bound is
therefore impossible while the smooth solution exists.
If its maximal existence time were finite, the established low bounds
would satisfy Lemma~\ref{lem:hp-persistence}. The high Sobolev norm
would remain finite, and the values would remain in a compact positive
phase region. The classical continuation theorem would extend the
solution, a contradiction. Every smooth aligned approximation is thus
global.

More explicitly, let $\eta_\varepsilon$ be a nonnegative smooth
approximate identity of mass one, and set
\[
 H_0^\varepsilon=\hb+\eta_\varepsilon*(H_0-\hb),\quad
 \rho_0^\varepsilon=(\rhob/\hb)H_0^\varepsilon,\quad
 u_0^\varepsilon=\eta_\varepsilon*u_0,\quad
 \theta_0^\varepsilon=\tb+\eta_\varepsilon*(\theta_0-\tb).
\]
These data belong to every $H^m$ perturbation class, preserve alignment
and positivity, and have the same bound $C_0$.
The corresponding solutions satisfy all the low-order estimates
uniformly in $\varepsilon$ on every finite interval.
The high $H^5$ estimate is used separately for each approximation;
its initial value need not be bounded as $\varepsilon\downarrow0$.

\subsection{Compactness and passage to the equations}
\label{subsec:global-compactness}

Fix $T<\infty$. In this subsection, constants may depend on the now
fixed $\tb,\hb$ and on $T$, but not on $\varepsilon$.
The estimates already proved give uniform bounds for
\[
 H^\varepsilon-\hb\text{ in }L^\infty(0,T;H^1\cap D^{1,4}),
 \quad
 (u^\varepsilon,\theta^\varepsilon-\tb)
                 \text{ in }L^\infty(0,T;H^2),
 \quad
 (\dot u^\varepsilon,\dot\theta^\varepsilon)
                 \text{ in }L^\infty L^2\cap L^2H^1 .
\]
The $L^4$ thermal estimate \eqref{eq:hp-thermal-L4} also shows
$(\nabla^2u^\varepsilon,\nabla^2\theta^\varepsilon)\in L^2L^4$
uniformly on this interval.
Since $u_t=\dot u-u\cdot\nabla u$ and
$\theta_t=\dot\theta-u\cdot\nabla\theta$, differentiation gives
\begin{align}\label{eq:global-time-derivatives}
 \|\nabla u_t\|_2
 &\le\|\nabla\dot u\|_2+C\|\nabla u\|_4^2
                                    +\|u\|_\infty\|\nabla^2u\|_2,\notag\\
 \|\nabla\theta_t\|_2
 &\le\|\nabla\dot\theta\|_2
       +\|\nabla u\|_4\|\nabla\theta\|_4
                                    +\|u\|_\infty\|\nabla^2\theta\|_2 .
\end{align}
Consequently $(u_t^\varepsilon,\theta_t^\varepsilon)$ is bounded in
$L^\infty L^2\cap L^2H^1$.
The transport equation gives
$H_t^\varepsilon=-u^\varepsilon\cdot\nabla H^\varepsilon
-H^\varepsilon\dv u^\varepsilon$, bounded in
$L^\infty(0,T;L^2\cap L^4)$.

Rellich compactness on bounded spatial balls, the preceding time
bounds and the Arzel\`a--Ascoli theorem give a subsequence such that,
locally in space,
\begin{equation}\label{eq:global-local-convergence}
 \begin{aligned}
 &(u^\varepsilon,\theta^\varepsilon-\tb)
       \longrightarrow(u,\theta-\tb)
                &&\text{in }C([0,T];H^s_{\rm loc}),\quad s<2,\\
 &H^\varepsilon-\hb\longrightarrow H-\hb
                &&\text{in }C([0,T];L^2_{\rm loc}).
 \end{aligned}
\end{equation}
The uniform $W^{1,4}$ bound upgrades the latter convergence locally
to uniform convergence in space and time. Choosing $s>3/2$
also gives local strong convergence of $\nabla u^\varepsilon$ and
$\nabla\theta^\varepsilon$ in $C_tL^4_x$.
All bounded coefficients consequently converge strongly locally;
the first derivatives of $H^\varepsilon$ and the second derivatives
of the parabolic variables converge weakly in their displayed spaces.

These convergences justify each equation in distributions.
In particular, the quadratic viscous heating converges strongly
locally in $L^2$, the variable-coefficient stresses and heat fluxes
converge in distributions, and the time terms converge by their
uniform bounds. Alignment passes to the limit exactly.
A diagonal subsequence over balls and integer time endpoints gives
a global solution. Weak lower semicontinuity preserves the global
bounds; no compact embedding on all of $\R^2$ is used.
The time-derivative bounds and the convergence of the initial data
identify the initial traces in $L^2$. We next prove continuity in
the stronger spaces asserted in the theorem.

\subsection{Strong time continuity, including the initial time}
\label{subsec:global-continuity}

First, $u,\theta-\tb\in C([0,T];H^s)$ for every $s<2$.
Indeed the $L^\infty_tL^2_x$ time derivatives give global Lipschitz
continuity into $L^2$, and interpolation with the uniform $H^2$
bound yields the assertion. In particular their values are continuous
in $L^\infty$ and their gradients in $L^4$.

The transport velocity belongs to $L^1(0,T;W^{1,\infty})$, because
$u\in L^\infty H^2$ and $\nabla^2u\in L^2L^4$.
Its flow is therefore bi-Lipschitz on each finite interval.
Differentiating the magnetic equation gives the linear equation
\[
 \Dt\nabla H
 =-(\nabla u)^{\top}\nabla H-(\dv u)\nabla H-H\nabla\dv u .
\]
The matrix coefficient belongs to $L^1L^\infty$ and the last term
belongs to $L^1(L^2\cap L^4)$. The flow formula, initially for smooth
functions and then by density in $L^2\cap L^4$, yields
$\nabla H\in C([0,T];L^2\cap L^4)$.
Together with $H-\hb\in C L^2$ this proves the transported part of
\eqref{eq:fixed-target-space}. It also gives $\rho\in C L^\infty$
after subtraction of its background.

For clarity, the material derivatives require a separate argument;
boundedness in $L^\infty L^2$ alone would not prove their continuity.
Put $w=\dot u$, $z=\dot\theta$.
The conservative material equations underlying
\eqref{eq:fdim-momentum-identity} and
\eqref{eq:fdim-temperature-identity} imply
\[
 (\sqrt\rho\,w)_t,\;(\sqrt\rho\,z)_t
                       \in L^2(0,T;H^{-1}).
\]
We verify the products in this statement. Both material derivatives
belong to $L^2H^1$, hence to $L^2L^4$; both spatial gradients
$\nabla u,\nabla\theta$ belong to $L^\infty L^4$.
Thus $z\nabla u$, $z\nabla\theta$ and the quadratic gradient
fluxes belong to $L^2L^2$. In the thermal source,
$zQ(u)$ and $\nabla u:\nabla w$ belong to
$L^2L^{4/3}\subset L^2H^{-1}$, and the cubic velocity-gradient
terms have the same property on $[0,T]$.
Transport fluxes are in $L^2L^2$.
Multiplication by $\rho^{-1/2}$ maps $H^1$ continuously into itself:
\[
 \|\nabla(\rho^{-1/2}\varphi)\|_2
 \le C\|\nabla\varphi\|_2+C\|\nabla\rho\|_4\|\varphi\|_4
 \le C\|\varphi\|_{H^1}.
\]
Finally, $\rho_t\in L^\infty L^4$, so the derivative of this
multiplier produces only $L^2H^{-1}$ terms.
Also $\sqrt\rho\,w,\sqrt\rho\,z\in L^2H^1$.

The Hilbert-triple continuity lemma
$L^2H^1\cap H^1H^{-1}\hookrightarrow C L^2$
therefore applies; see \citep[Chapter~I]{LionsMagenes1972}.
It follows that $w,z\in C([0,T];L^2)$.
Their initial traces satisfy
\[
 \sqrt{\rho_0}\,w(0)=-g_1,\qquad
 \sqrt{\rho_0}\,z(0)=-g_2,
\]
with $g_1,g_2$ defined in \eqref{eq:fixed-compatibility}.
To justify the traces explicitly, the same product estimates give
uniform bounds for
$Y^\varepsilon=\sqrt{\rho^\varepsilon}\,\dot u^\varepsilon$
and $Z^\varepsilon=\sqrt{\rho^\varepsilon}\,\dot\theta^\varepsilon$
in $L^2H^1\cap H^1H^{-1}$. Their weak limits in that Hilbert space
are $\sqrt\rho\,w$ and $\sqrt\rho\,z$, by the local convergences above.
Evaluation at $t=0$ is a continuous linear map from this space into
$L^2$. Since $Y^\varepsilon(0)=-g_1^\varepsilon$ and
$Z^\varepsilon(0)=-g_2^\varepsilon$, with
$g_i^\varepsilon\to g_i$ strongly in $L^2$, the displayed initial
identities follow. No extra compatibility hypothesis has been introduced.

It remains to recover the highest spatial derivatives continuously.
The expanded thermal equation is
\[
 \kappa\Delta\theta
 =\rho\theta^{-\beta}z+R\rho\theta^{1-\beta}\dv u
       -\theta^{\alpha-\beta}Q(u)
       -\kappa\beta\theta^{-1}|\nabla\theta|^2 .
\]
Every term on the right belongs to $C([0,T];L^2)$ by the
continuities just proved. The $L^2$ elliptic estimate therefore
gives $\theta-\tb\in C H^2$.
Likewise,
\[
 \mathcal L u=\rho\theta^{-\alpha}w
       +\theta^{-\alpha}\nabla P_c
       -\alpha\theta^{-1}S_0(u)\nabla\theta
\]
has a continuous $L^2$ right-hand side: use
$\nabla\rho,\nabla H\in C L^2$ and the continuous $L^4$ gradients.
The Lam\'e estimate gives $u\in C H^2$.
This proves all continuity assertions in
\eqref{eq:fixed-target-space}, including at $t=0$.

\subsection{Uniqueness and completion of the proof}
\label{subsec:global-uniqueness}

Take two solutions in the displayed strong class with the same
initial data and backgrounds, and write
$\delta H=H_1-H_2$, $\delta u=u_1-u_2$ and
$\delta\theta=\theta_1-\theta_2$.
Alignment gives $\delta\rho=(\rhob/\hb)\delta H$.
On $[0,T]$ both solutions have finite $H^2$ and $W^{1,4}$ bounds,
and positive density and temperature bounds. Constants in this
uniqueness argument may depend on these bounds and the fixed
backgrounds.

The difference of the magnetic equations is
\[
 (\partial_t+u_1\cdot\nabla)\delta H
 =-H_1\dv\delta u-\delta u\cdot\nabla H_2
                       -\delta H\,\dv u_2 .
\]
Test it by $\delta H$; test the momentum and thermal differences
by $\delta u,\delta\theta$, respectively, using the mass weight
$\rho_1$. The pressure difference is integrated by parts and obeys
$\|P_{c,1}-P_{c,2}\|_2
 \le C(\|\delta H\|_2+\|\delta\theta\|_2)$.
The principal viscous and thermal quadratic forms are coercive.

We give the estimates for the terms where regularity is decisive.
For $f=\dot u_2$ or $\dot\theta_2$, and $v=\delta u$ or
$\delta\theta$, respectively,
\[
 \int|\delta H||f||v|
 \le\|\delta H\|_2\|f\|_4\|v\|_4
 \le\varepsilon\|\nabla v\|_2^2+
 C_\varepsilon\|f\|_4^{4/3}
                     (\|\delta H\|_2^2+\|v\|_2^2).
\]
The transport term containing $\nabla H_2$ is estimated by the
same inequality with $f=\nabla H_2$ and $v=\delta u$.
Temperature-dependent stress differences satisfy
\begin{align*}
 \int|\delta\theta||\nabla u_2||\nabla\delta u|
 &\le\|\delta\theta\|_4\|\nabla u_2\|_4\|\nabla\delta u\|_2
 \le\varepsilon(\|\nabla\delta u\|_2^2+
                           \|\nabla\delta\theta\|_2^2)
             +C_\varepsilon\|\nabla u_2\|_4^4\|\delta\theta\|_2^2.
\end{align*}
The heat-flux difference has the same bound with $\nabla\theta_2$.
For the difference of viscous heating use
\[
 |Q(u_1)-Q(u_2)|
 \le C(|\nabla u_1|+|\nabla u_2|)|\nabla\delta u|.
\]
The resulting product with $\delta\theta$ is controlled by the
preceding two Young inequalities. The coefficient difference in
the heating term is bounded using
$\|\delta\theta\|_4^2\|\nabla u_2\|_4^2$.
All convection and thermal pressure terms are bounded by the
same interpolations or by the integrable Lipschitz norm of $u_i$.

After absorbing the finitely many small derivative terms, we obtain
\begin{equation}\label{eq:global-uniqueness}
 \begin{aligned}
 &\frac{\mathrm d}{\mathrm dt}
  \left[\|\delta H\|_2^2+
       \int\rho_1(|\delta u|^2+|\delta\theta|^2)\dd x\right]
       +c(\|\nabla\delta u\|_2^2+
                         \|\nabla\delta\theta\|_2^2)\\
 &\le G(t)
  \left[\|\delta H\|_2^2+
       \int\rho_1(|\delta u|^2+|\delta\theta|^2)\dd x\right],
 \end{aligned}
\end{equation}
where one may take
\[
 G(t)=C_T\left[1+\|\nabla u_1\|_\infty+\|\nabla u_2\|_\infty
          +\|\dot u_2\|_4^{4/3}
          +\|\dot\theta_2\|_4^{4/3}\right]\in L^1(0,T).
\]
Gronwall's inequality proves uniqueness. The arbitrariness of $T$
gives uniqueness globally.

The bootstrap bounds give \eqref{eq:global-uniform}; the basic,
first-order, material, elliptic and magnetic dissipation estimates
give \eqref{eq:global-integrals}, first on $[0,T]$ and then on
$[0,\infty)$ by monotone convergence. This completes the proof of
Theorem~\ref{thm:global}.

\end{document}